\documentclass[a4paper,10pt]{article}

\usepackage{hyperref}
\usepackage{color}
\usepackage[english]{babel}
\usepackage[utf8]{inputenc}
\usepackage[T1]{fontenc}
\usepackage{setspace}

\usepackage{enumerate}
\usepackage{microtype}
\usepackage{cite}
\usepackage[a4paper,top=2cm,bottom=2cm,left=2.5cm,right=2.5cm,bindingoffset=5mm]{geometry}

\usepackage{tikz}
\usetikzlibrary{decorations.pathmorphing}

\usepackage{xfrac,stmaryrd,mathrsfs,bm,mathtools,yfonts}
\usepackage{amsmath}
\usepackage{amsfonts}
\usepackage{amssymb}
\usepackage{amsthm}
\usepackage{braket}
\usepackage{marvosym}
\usepackage{eufrak}
\usepackage{esint}

\theoremstyle{plain}
\newtheorem{teo}{Theorem}[section]
\newtheorem{lem}[teo]{Lemma}
\newtheorem{cor}[teo]{Corollary}
\newtheorem{prop}[teo]{Proposition}
\newtheorem*{teo*}{Theorem}

\theoremstyle{remark}
\newtheorem{oss}[teo]{Remark}

\theoremstyle{definition}
\newtheorem{defi}[teo]{Definition}

\numberwithin{equation}{section}

\newcommand{\N}{\mathbb{N}}
\newcommand{\R}{\mathbb{R}}
\newcommand{\C}{\mathbb{C}}
\newcommand{\Sf}{\mathbb{S}}

\renewcommand{\a}{\alpha}
\renewcommand{\b}{\beta}
\renewcommand{\l}{\lambda}
\renewcommand{\L}{\Lambda}
\renewcommand{\d}{\delta}
\renewcommand{\c}{\gamma}
\newcommand{\s}{\sigma}

\newcommand{\A}{\mathbb{A}}
\newcommand{\LL}{\mathbb{L}}
\newcommand{\Om}{\Omega}
\newcommand{\eps}{\varepsilon}
\newcommand{\diw}{\mathrm{div}}

\renewcommand{\phi}{\varphi}

\title{The Classical Obstacle problem with coefficients in fractional Sobolev spaces}
\author{Francesco Geraci}
\date{}
\begin{document}
\maketitle

\begin{abstract}
We prove quasi-monotonicity formulae for classical obstacle-type problems with
quadratic energies with coefficients in fractional Sobolev spaces, and a linear term
with a Dini-type continuity property. 
These formulae are used to obtain the regularity of free-boundary points following the approaches by
Caffarelli, Monneau and Weiss.
\end{abstract}

\section{Introduction}
The motivation for studying obstacle problems has roots in many applications. There are examples in physics and in mechanics,
and many prime examples can be found in \cite{Chipot,Fried,KO,KS,Rod}.
The classical obstacle problem consists in finding the minimizer of Dirichlet energy
in a domain $\Om$, among all functions $v$, with fixed boundary data, constrained to lie above a given obstacle $\psi$. 
Active areas of research related to this problem include both studying properties of minimizers and analyzing 
the regularity of the boundary of the coincidence set between minimizer and obstacle.

In the '60s the obstacle problem was introduced within the study of variational inequalities. 
In the last fifty years much effort has been made to understand of the problem, a wide variety of issues has been 
analyzed and new mathematical ideas have been introduced. Caffarelli in \cite{Caf80} introduced the so-called method of 
\emph{blow-up}, imported from geometric measure theory for the study of minimal surfaces, to prove some local properties of 
solutions. Alt-Caffarelli-Friedman \cite{ACF}, Weiss \cite{Weiss} and Monneau \cite{Monneau} introduced monotonicity formulae 
to show the blow-up property and to obtain the free-boundary regularity in various problems. 
See \cite{Caf77,CS05,Fried,KS,PSU,Rod} for more detailed references and historical developments.

Recently many authors have improved classical results replacing the Dirichlet energy by a more general variational
functional and weakening the regularity of the obstacle (we can see  \cite{CFS04,FS07,Foc10,Foc12,FGS,MP11,Mon09,Wang00,Wang02}). 
In this context, we aim to minimize the following energy
\begin{equation}
 \mathcal{E}(v):=\int_\Om \big(\langle \A(x)\nabla v(x),\nabla v(x)\rangle +2f(x)v(x)\big)\,dx,
\end{equation}
among all positive functions with fixed boundary data, where
$\Omega\subset \R^n$ is a smooth, bounded and open set, $n\geq 2$, $\mathbb{A}:\Omega\to \R^{n\times n}$ is 
a matrix-valued field and $f:\Omega\to \R$ is a function satisfying:
\begin{itemize}
\item[$(H1)$]$\mathbb{A}\in W^{1+s,p}(\Omega;\R^{n\times n})$ with $s>\frac{1}{p}\,$ and $\,p>\frac{n^2}{n(1+s)-1} \wedge n$;
where the symbol $\wedge$ indicates the minimum of the surrounding quantities;
\item[$(H2)$] $\mathbb{A}(x)=\left(a_{ij}(x)\right)_{i,j=1,\dots,n}$ symmetric, continuous and coercive, 
  that is $a_{ij}=a_{ji}$ $\mathcal{L}^n$ a.e. $\Om$  and for some $\L\geq 1$ i.e.
  \begin{equation}\label{A coerc cont}
   \Lambda^{-1}|\xi|^2\leq \langle \mathbb{A}(x)\xi, \xi\rangle \leq \Lambda|\xi|^2 \qquad\qquad \mathcal{L}^n \,\,\textrm{a.e.}\,\,\Om, \,\,\forall \xi\in \R^n;
  \end{equation}
 \item[$(H3)$] $f$ Dini-continuous, that is $\omega(t)=\sup_{|x-y|\leq t} |f(x)-f(y)|$ modulus of continuity $f$ satisfying the following
 integrability condition
 \begin{equation}\label{e:Dini continuity}
  \int_0^1 \frac{\omega(t)}{t}\,dt < \infty,
 \end{equation}
and exists $c_0>0$ such that $f\geq c_0$.
 \item[$(H4)$] $f$ satisfies a double Dini-type condition:  
 let $\omega(t)=\sup_{|x-y|\leq t} |f(x)-f(y)|$ be the modulus of continuity of $f$ and set $a\geq 1$ it holds 
 the following condition of integrability
  \begin{equation}\label{H3'}
    \int_0^1 \frac{\omega(r)}{r}\,|\log r|^a\, dr < \infty,
  \end{equation}
  and exists $c_0>0$ such that $f\geq c_0$.
\end{itemize}

In Remark \ref{choice (H1)} we will justify the choice of $p$ in hypothesis $(H1)$.
We note that we are reduced to the $0$ obstacle case, so $f=-\diw(\A\nabla \psi)$.

In the paper we prove that the unique minimizer, which we will indicate as $u$, 
is the solution of an elliptic differential equation in divergence form. With classical PDE regularity theory 
we deduce that $u$ and $\nabla u$ are H\"{o}lder continuous and $\nabla^2 u$ is integrable.
To prove the regularity of the free-boundary $\Gamma_u=\partial \{u=0\}\cap \Omega$, we apply the method of blow-up 
introduced by Caffarelli \cite{Caf80}.
For all $x_0$ points of free-boundary $\Gamma_u=\partial \{u=0\}\cap \Omega$ we introduce a sequence of rescaled functions and, 
through a $C^{1,\gamma}$ estimate of rescaled functions (for a suitable $\gamma\in (0,1)$), 
we prove the existence of sequence limits; these limits are called blow-ups. 
To classify the blow-ups and to prove the uniqueness of the sequence limit, for all points of $\Gamma_u$, we introduce a technical tool: the 
quasi-monotonicity formulae. 
To simplify the notation we introduce, for all $x_0\in\Gamma_u$ an opportune change of variable for which, 
without loss of generality, we can suppose:
\begin{equation}\label{abuso_notazione}
  x_0=\underline{0}\in \Gamma_u,\qquad \A(\underline{0})=I_n,\qquad f(\underline{0})=1.
\end{equation}
As in \cite{FGS} we introduce the auxiliary energy ``à la Weiss'' 
\begin{equation*}
 \Phi(r):=\int_{B_1}\big(\langle \A(rx)\nabla u_r(x),\nabla u_r(x)\rangle +2f(rx)u_r(x)\big)\,dx + 2\int_{\partial B_1}\left\langle \A(rx)\frac{x}{|x|},\frac{x}{|x|}\right\rangle u_r^2(x) \,d\mathcal{H}^{n-1}(x),
\end{equation*}
and prove the main results of the paper:

\begin{teo}[Weiss' quasi-monotonicity formula]\label{Weiss}
 Assume that $(H1)$-$(H3)$ and \eqref{abuso_notazione} hold. There exist nonnegative constants $\bar{C_3}$ and $C_4$ independent
 from $r$ such that the function
 \begin{equation*}
  r\mapsto \Phi(r)\,e^{\bar{C_3}r^{1-\frac{n}{\Theta}}} + C_4\,\int_0^r \left(t^{-\frac{n}{\Theta}}+\frac{\omega(t)}{t}\right)e^{\bar{C_3}t^{1-\frac{n}{\Theta}}}\,dt
 \end{equation*}
with the constant $\Theta$ given in equation \eqref{B}, is nondecreasing on the interval $(0,\frac{1}{2}\mathrm{dist}(\underline{0},\partial\Om)\wedge 1)$.

More precisely, the following estimate holds true for $\mathcal{L}^1$-a.e. $r$ in such an interval:
\begin{equation}\label{dis monotonia Weiss}
 \begin{split}
 \frac{d}{dr}\bigg(\Phi(r)\,e^{\bar{C_3}r^{1-\frac{n}{\Theta}}} &+ C_4\,\int_0^r \left(t^{-\frac{n}{\Theta}}+\frac{\omega(t)}{t}\right)e^{\bar{C_3}t^{1-\frac{n}{\Theta}}}\,dt\bigg)\\
 &\geq \frac{2e^{\bar{C_3}r^{1-\frac{n}{\Theta}}}}{r^{n+2}}\int_{\partial B_r}\mu \Big(\langle \mu^{-1} \A\nu, \nabla u \rangle - 2\frac{u}{r} \Big)^2\,d\mathcal{H}^{n-1}.
 \end{split}
\end{equation}
In particular, the limit $\Phi(0^+):=\lim_{r\to 0^+}\Phi(r)$  exists and it is finite and there exists a constant $c>0$ such that
\begin{equation}\label{weiss stima Phir}
 \begin{split}
 &\Phi(r)-\Phi(0^+)\\
 &\phantom{A}\geq \Phi(r)\,e^{\bar{C_3}r^{1-\frac{n}{\Theta}}} + C_4\,\int_0^r \left(t^{-\frac{n}{\Theta}}+\frac{\omega(t)}{t}\right)e^{\bar{C_3}t^{1-\frac{n}{\Theta}}}\,dt -\Phi(0^+) - c\,\left(r^{1-\frac{n}{\Theta}}+\int_0^r\frac{\omega(t)}{t}\,dt\right).
 \end{split}
\end{equation}
\end{teo}

\begin{teo}[Monneau's quasi-monotonicity formula]\label{Monneau}
 Assume $(H1)$, $(H2)$ and $(H4)$ with $a\geq 1$ and \eqref{abuso_notazione}. 
 Let $u$ be the minimizer of $\mathcal{E}$ on $K$, with
 $\underline{0}\in Sing(u)$ (i.e. \eqref{pto sing1} holds), and $v$ be a $2$-homogeneous, positive, polynomial function, 
 solution of $\Delta v=1$ on $\R^n$. Then, there exists a positive constant 
 $C_5=C_5(\l,\|\A\|_{W^{s,p}})$  such that
 \begin{equation}\label{funz monotona Monneau}
  r \longmapsto \int_{\partial B_1} (u_r-v)^2\,d\mathcal{H}^{n-1} + C_5\,\left(r^{1-\frac{n}{\Theta}}+\int_0^r\frac{\omega(t)}{t}\,dt + \int_0^r\frac{dt}{t}\int_0^t\frac{\omega(s)}{s}\,ds\right)
 \end{equation}
is nondecreasing on  $(0,\frac{1}{2}\mathrm{dist}(\underline{0},\partial\Om)\wedge 1)$. More precisely, $\mathcal{L}^1$-a.e. 
on such an interval
\begin{equation}\label{dis monotonia Monneau}
\begin{split}
 \frac{d}{dr}\bigg(\int_{\partial B_1}(u_r-v)^2\,&d\mathcal{H}^{n-1} + C_5\,\left(r^{1-\frac{n}{\Theta}}+\int_0^r\frac{\omega(t)}{t}\,dt + \int_0^r\frac{dt}{t}\int_0^t\frac{\omega(s)}{s}\,ds \right)\bigg)\\
 &\geq \frac{2}{r}\bigg(e^{C_3\,r^{1-\frac{n}{\Theta}}}\Phi(r) + C_4 \int_0^r e^{c_3t^{1-\frac{n}{\Theta}}}\left(t^{-\frac{n}{\Theta}}+\frac{\omega(t)}{t}\right)\,dt-\Psi_v(1)\bigg).
 \end{split}
\end{equation}
where $\Psi_v(1)=\int_{B_1}\big(|\nabla v|^2 + 2v\big)\,dx -2\int_{\partial B_1} v^2\,d\mathcal{H}^{n-1}$
\end{teo}

These theorems
generalize the results of Weiss \cite{Weiss} and Monneau \cite{Monneau}. 
The Weiss monotonicity formula was proven by Weiss within \cite{Weiss} for the case where $\A\equiv I_n$ and $f\equiv 1$; 
in the same paper he proved the celebrated epiperimetric inequality (see Theorem \ref{epip Weiss}) and gave a new way 
of approaching the problem of the regularity for the free-boundary. 
In \cite{PS07} Petrosyan and Shahgholian proved the monotonicity formula for $\A\equiv I_n$ and $f$ with a double Dini 
modulus of continuity (but for obstacle problems with no sign condition on the solution). 
Lederman and Wolanski \cite{LW07} provided a local monotonicity formula for the perturbated problem to achieve the regularity of
Bernoulli and Stefan free-boundary problems, while Ma, Song and Zhao \cite{MSZ10} showed the formula for elliptic and parabolic 
systems in the case in which $\A\equiv I_n$ and the equations present a first order nonlinear term. 
Garofalo and Petrosyan in \cite{GP09} proved the formula for the thin obstacle problem with a smooth obstacle. 

Garofalo, Petrosyan and Smit Vega Garcia in \cite{GPS16} proved the result for Signorini's problem under the 
hypotheses $\A\in W^{1,\infty}$ and $f\in L^\infty$.
Focardi, Gelli and Spadaro in \cite{FGS} proved the formula for the classical obstacle problem for $\A\in W^{1,\infty}$ and $f\in C^{0,\alpha}$
for $\alpha\in (0,1)$. In the same paper (under the same hypotheses of coefficients) the three authors proved a generalization 
of the monotonicity formula introduced by Monneau \cite{Monneau} to analyze the behaviour near the singular points (see Definition \ref{reg sing}).
In \cite{Mon09} he improved his result; he showed that his monotonicity formula holds
under the hypotheses that $\A\equiv I_n$ and $f$ with a Dini modulus of continuity in an $L^p$ sense. 
In \cite{GP09} Garofalo and Petrosyan showed the formula of Monneau for the thin obstacle with a regular obstacle.\\

In our work (inspired by \cite{FGS}) we prove the quasi-monotonicity formulae under the hypotheses, $(H1)$-$(H4)$ 
improving the results with respect to current literature.
As we will see in Corollary \ref{cor Morrey fraz} if $ps>n$ the embedding $W^{1+s,p}\hookrightarrow W^{1,\infty}$ holds true.
Consequently, we assume $sp\leq n$ and we obtain an original result not covered by \cite{FGS} if $p>\frac{n^2}{n(1+s)-1}\wedge n$. 
(We can observe that $(\frac{n^2}{n(1+s)-1}\wedge n )< \frac{n}{s}$ for all $s\in \R$.)\\ 

Weiss' quasi-monotonicity formula allows us first to deduce that blow-ups are homogeneous of degree $2$, and second 
(using also the nondegeneracy of the solution proven in an even more general setting by Blank and Hao \cite{Blank})
to show that the blow-ups are nonzero.
Thanks to a $\Gamma$-convergence argument and according to Caffarelli's classification of blow-ups, in the classical case
(see \cite{Caf77,Caf80,Caf98}), we can classify the blow-up types and so distinguish the points in $\Gamma_u$ 
as regular and singular (respectively $Reg(u)$ and $Sing(u)$, see Definition \ref{reg sing}).

Following the energetic approach by Focardi, Gelli and Spadaro \cite{FGS}
we prove the uniqueness of blow-ups for both the regular and the singular cases.
In the classical framework, the uniqueness of the blow-ups can be derived, a posteriori,
from the regularity properties of the free-boundary (see Caffarelli \cite{Caf80}).
In our setting we distinguish two cases: $x_0\in Sing(u)$ and $x_0\in Reg(u)$.
In the first case, through the two quasi-monotonicity formulae and an ``absurdum'' argument, we prove the uniqueness of blow-ups 
providing a uniform decay estimate for all points in a compact subset of $Sing(u)$.
In the second case, we need to introduce an assumption, probably of a technical nature, on the
modulus of continuity of $f$: $(H4)$ with $a>2$ 
(more restrictive than double Dini continuity which is equivalent to $(H4)$ with $a=1$, see \cite[Definition 1.1]{Mon09}).

So, thanks to the epiperimetric inequality of Weiss \cite{Weiss} we obtain a uniform decay estimate for the convergence
of the rescaled functions with respect to their blow-up limits.
We recall that Weiss \cite{Weiss} proved the uniqueness for regular points in $\A\equiv I_n$ and $f\equiv 1$.
Focardi, Gelli and Spadaro \cite{FGS} also had proved our same result for $\A$ Lipschitz continuous and $f$ H\"{o}lder continuous.
Monneau \cite{Mon09} proved the uniqueness of blow-ups both for regular points and for singular points with $\A\equiv I_n$ and $f$ with Dini continuous 
modulus of mean oscillation in $L^p$. Therefore, without further hypotheses, in the regular case and adding
double Dini continuity condition on the modulus of the mean oscillation, Monneau gave a very accurate pointwise decay estimate, 
providing an explicit modulus of continuity for the solution.

These results allow us to prove the regularity of free-boundary:
\begin{teo} 
We assume the hypothesis $(H1)$-$(H3)$. The free-boundary decomposes as $\Gamma_u=Reg(u)\cup Sing(u)$ with $Reg(u)\cap Sing(u)=\emptyset$.  
\begin{itemize}
 \item[(i)] Assume $(H4)$ with $a>2$. $Reg(u)$ is relatively open in $\partial \{u=0\}$ and for every point $x_0\in Reg(u)$. 
	    there exists $r=r(x_0)>0$ such that $\Gamma_u\cap B_r(x_0)$ 
	    is a $C^1$ hypersurface with normal vector $\varsigma$ is absolutely continuous with a modulus of continuity depending 
	    on $\rho$ defined in \eqref{rho}.
	    
	    In particular if $f$ is H\"{o}lder continuous there exists $r=r(x_0)>0$ such that $\Gamma_u\cap B_r(x)$ 
	    is $C^{1,\b}$ hypersurface for some universal exponent $\b \in (0,1)$.
 \item[(ii)] Assume $(H4)$ with $a\geq 1$. $Sing(u)=\cup_{k=0}^{n-1} S_k$ (see Definition \ref{d:singular stratum})	     
	     and for all $x\in S_k$ there 
	     exists $r$ such that $S_k\cap B_r(x)$ is contained in a regular $k$-dimensional submanifold of $\R^n$.
\end{itemize}
\end{teo}

In order to justify the choice of regularity of the coefficients of $\A$ and $f$ we discuss the hypotheses $(H1)$ and $(H3)$.

The hypothesis $(H3)$ turns out to be the best condition to obtain the uniqueness of blow-up 
(we need of $(H4)$ with $a\geq 1$ in the case of singular points and $a>2$ in the case of regular points). 
In fact when condition \eqref{e:Dini continuity} is not satisfied, 
Blank gave in \cite{Bl01} an example of nonuniqueness of the blow-up limit at a
regular point. Monneau observed in \cite{Monneau} that using the symmetry $x \mapsto - x$, 
it is easy to transform the result of Blank into an example
of nonuniqueness of the blow-up limit at a singular point when condition \eqref{e:Dini continuity} is not satisfied.

Before taking into account hypothesis $(H1)$ we need to clarify the relationship between the regularity of coefficients $\A$, $f$ and 
the regularity of the free-boundary. 
Caffarelli \cite{Caf77} and Kinderlehrer and Nirenberg \cite{KN77} proved that for smooth coefficients
of $\A$ and for $f\in C^1$ the regular points are a $C^{1,\alpha}$-manifold for all $\a\in (0,1)$, for $f\in C^{m,\alpha}$ 
$Reg(u)$ is a $C^{m+1,\alpha}$-manifold with $\alpha\in (0,1)$ and if $f$ is analytic so is $Reg(u)$.
In \cite{Bl01} Blank proved that, in the Laplacian case with $f$ Dini continuous, the set of regular points is a $C^1$-manifold, 
but if $f$ is $C^0$, but is not Dini continuous, then $Reg(u)$ is Reifenberg vanishing, but not necessarily smooth. 
In \cite{FGS} Focardi, Gelli and Spadaro proved that if $\A\in W^{1,\infty}$ and $f\in C^{0,\alpha}$ with $\alpha\in (0,1)$,
then $Reg(u)$ is a $C^{1,\beta}$-manifold with $\beta\in (0,\alpha)$. 
A careful inspection of the proof of \cite[Theorem 4.12]{FGS} shows that in the case of $\A\in W^{1,\infty}$ and $f\equiv 1$ the 
regular set turns out to be a $C^{1,\beta'}$-manifold with $\beta'\in (0,\frac{1}{2})$, so, despite the linear term being constant,
the regularity improves slightly but remains in the same class. 
Blank and Hao in \cite{BH15} proved that if $a_{i,j},f\in VMO$, any compact set $K\subset\subset Reg(u)\cap B_\frac{1}{2}$
is relatively Reifenberg vanishing with respect to $Reg(u)\cap B_\frac{1}{2}$.
So the regularity of the regular part of the free-boundary turns out to be strictly related to regularity of coefficients 
of matrix $\A$ and the linear term $f$. 
Under the hypotheses $(H1)$ and $(H2)$ we prove that 
if $f$ is H\"{o}lder continuous we obtain that the 
regular part of the free-boundary is a $C^{1,\beta}$-manifold for some $\beta$, while if $f$ satisfies hypothesis $(H4)$ with $a>2$ 
we prove that $Reg(u)$ is a $C^1$-manifold.

So the process of weakening the regularity of coefficients goes along two directions: to obtain a strong  
or a weak regularity of the regular part of the free-boundary. 
Our work forms part of the first way and with the technical hypothesis $(H4)$ wiht $a>2$ for $f$, 
which is better than the H\"{o}lder continuity, 
and by hypothesis $(H1)$ of matrix $\A$, we improve the current literature. 
The best regularity for $\A$  that allows us to have a strong regularity of $Reg(u)$ still remains, 
to our knowledge, an open problem.
Regarding the best regularity for $f$, from \cite{Bl01}  we know that it is the Dini continuity; we do not reach it but we improve
the already investigated condition of H\"{o}lder continuity.

The natural sequel of these results is the study of obstacle problems for nonlinear energies. The future developments 
aim at the same direction as \cite{FGerS} where the author, Focardi and Spadaro prove an exhaustive analysis 
of the free-boundary for nonlinear variational energies as the outcome of analogous results for the classical obstacle problem 
for quadratic energies with Lipschitz coefficients.\\

To conclude, the paper is organized as follows. In Section $2$ we fix the notation of fractional Sobolev spaces. 
In Section $3$ we prove the existence, the uniqueness and the regularity of the minimizer $u$.
In Section $4$ we introduce the sequence of rescaled functions, prove the existence of blow-ups and state a property of 
nondegeneracy of the solution of the obstacle problem. 
In Section $5$ and $7$ we respectively prove the quasi-monotonicity formulae of Weiss and Monneau.
In Section $6$ we prove the $2$-homogeneity and the nonzero value property of blow-ups, 
classify blow-ups and distinguish the point of the free-boundary in regular and singular.
In Section $8$ we deduce the uniqueness of blow-ups in the case of regular and singular points.
In Section $9$ we state the properties of the regularity of the free-boundary.

\section{Preliminaries: Fractional Sobolev spaces}
In order to fix the notation, we report the definition of the fractional Sobolev spaces. 
See \cite{NPV,LM} for more detailed references.
\begin{defi}
 For any real $\l\in (0,1)$ and for all $p\in (0,\infty)$ we define the space 
 \begin{equation}
   W^{\l,p}(\Om):=\left\{v\in L^p(\Om) : \frac{|v(x)-v(y)|}{|x-y|^{\frac{n}{p}+\l}}\in L^p(\Om\times\Om) \right\},  
 \end{equation}
 i.e, an intermediary Banach space between $L^p(\Om)$ and $W^{1,p}(\Om)$, endowed with the norm
\begin{equation*}
 \|v\|_{W^{\l,p}(\Om)}=\left(\int_\Om |v|^p\,dx +\iint_{\Om\times\Om} \frac{|v(x)-v(y)|^p}{|x-y|^{n+\l p}}\,dx\,dy\right)^\frac{1}{p}.
\end{equation*}
If $\l>1$ and not integer we indicate with $\lfloor \l \rfloor$ its integer part and with $\sigma=\l - \lfloor \l \rfloor$  its fractional part. In this case the space $W^{\l,p}$ consists 
of functions $u\in W^{\lfloor \l \rfloor,p}$ such that the distributional derivatives $D^\a v\in W^{\s,p}$ with $|\a|=\lfloor \l \rfloor$
\begin{equation*}
  W^{\l,p}(\Om):=\left\{v\in W^{\lfloor \l \rfloor,p}(\Om) : \frac{|D^\a v(x)-D^\a v(y)|}{|x-y|^{\frac{n}{p}+\s}}\in L^p(\Om\times\Om),\quad \forall\a \,\,\textit{such that}\,\, |\a|=\lfloor \l \rfloor  \right\}.
\end{equation*}
$W^{\l,p}(\Om)$ is a Banach space with the norm
\begin{equation*}
 \|v\|_{W^{\l,p}(\Om)}=\left(\|v\|^p_{W^{\lfloor \l \rfloor,p}(\Om)} +\sum_{|\a|=\lfloor \l \rfloor}\|D^\a v\|^p_{W^{\s,p}(\Om)}\right)^\frac{1}{p}.
\end{equation*}
\end{defi}

We state three results on fractional Sobolev spaces useful for the follows. Theorems \ref{cor imm fraz} and \ref{toerema-traccia-fraz}
are proved, respectively in \cite{Mu} and \cite{Sch}, for Besov spaces; thanks to \cite[Remark 3.6]{Sch} 
and \cite[Theorem 14.40]{Leoni} we can reformulate these results in our notation.
Theorem \ref{cor Morrey fraz} is obtained combining classical Morrey theorem, \cite[Theore 8.3]{NPV} and Theorem \ref{cor imm fraz}.

\begin{teo}[Embedding Theorem {\cite[Theorem 9]{Mu} and \cite[Theorem 6.5]{NPV}}]\label{cor imm fraz}
Let $v\in W^{\l,p}(\Om)$ with $\l>0$, $1 <p<\infty$, $p\l<n$ and $\Om \subset \R^n$ be a bounded open set of class $C^{0,1}$.   
Then for all $0<t<\l$ 
there exists a constant $C=C(n,\l,p,t,\Om)$ for which 
\begin{equation*}
 \|v\|_{W^{t,\frac{np}{n-(\l-t)p}}(\Om)} \leq C\, \|v\|_{W^{\l,p}(\Om)}.
\end{equation*}
If $t=0$ there exists a constant $C=C(n,\l,p,\Om)$ for which 
\begin{equation*}
 \|v\|_{L^{\frac{np}{n-\l p}}(\Om)} \leq C\, \|v\|_{W^{\l,p}(\Om)}.
\end{equation*}
\end{teo}

\begin{teo}[Embedding Theorem]\label{cor Morrey fraz}
 Let $p\in [1,\infty)$ such that $sp>n$ and $\Om\subset\R^n$ be an extension domain 
 for $W^{\l,p}$. Then there exists a positive constant $C=C(n,p,\l,\Om)$, for which  
 \begin{equation}
  \|v\|_{C^{h,\a}}\leq C\|v\|_{W^{\l,p}(\Om)},
 \end{equation}
for all $v\in L^p(\Om)$ for some $\a\in (0,1)$ and $h$ integer with $h\leq m$. 
\end{teo}

\begin{teo}[Trace Theorem {\cite[Theorem 3.16]{Sch}}]\label{toerema-traccia-fraz}
Let $n\geq 2$, $0<p<\infty$, $\l>\frac{1}{p}$ and $U$ a bounded $C^k$ domain, $k>\l$ in $\R^n$. 
Then there exists a bounded operator
 \begin{equation}
  \c_0:W^{\l,p}(U)\longrightarrow W^{\l-\frac{1}{p},p}(\partial U; \mathcal{H}^{n-1}),
 \end{equation}
such that $\c_0(v)=v_{|\partial U}$ for all functions $v\in W^{\l,p}(U)\cap C(\overline{U})$. $\c_0$ is called \emph{trace operator}.
\end{teo}

\begin{oss}\label{oss stime imm e traccia}
 Let $p,\l$ be exponents as in theorem \ref{toerema-traccia-fraz}, $p\l<n$ and $\s:=\l-\lfloor \l \rfloor$. 
If $U=B_r$, we see how the constant of the trace operator changes  when the radius $r$ changes.
 
 By taking into account Theorem \ref{toerema-traccia-fraz} and \ref{cor imm fraz} we have the following embeddings
 \begin{equation*}
  W^{\l,p}(B_r)\hookrightarrow W^{\l-\frac{1}{p},p}(\partial B_r; \mathcal{H}^{n-1}) \hookrightarrow L^1(\partial B_r; \mathcal{H}^{n-1}).
 \end{equation*}
Then, setting $v_r(y)=v(ry)$
\begin{equation*}
\begin{split}
 \int_{\partial B_r}&|\c_0(v)(x)|\,d\mathcal{H}^{n-1} \stackrel{y=rx}{=}r^{n-1}\int_{\partial B_1}|\c_0(v_r)(y)|\,d\mathcal{H}^{n-1}\\
 &\leq c(1) r^{n-1}\bigg(\sum_{|\a|\leq \lfloor \l \rfloor}\int_{B_1}|D^\a v_r(y)|^p\,dx + \iint_{B_1\times B_1}\frac{|v_r(y)-v_r(z)|^p}{|y-z|^{n+\s p}}\,dy\,dz\bigg)^\frac{1}{p}\\
 &\stackrel{x=\frac{y}{r}}{=} c(1) r^{n-1}\,r^{-\frac{n}{p}} \bigg(\sum_{|\a|\leq \lfloor \l \rfloor}\int_{B_r}|D^\a v(x)|^p\,dx+r^{\s p}\iint_{B_r\times B_r}\frac{|v(x)-v(w)|^p}{|x-w|^{n+\s p}}\,dx\,dw\bigg)^\frac{1}{p}\\
 &\leq C r^{n-1}\,r^{-\frac{n}{p}} \|v\|_{W^{\l,p}(B_r)}.
\end{split}
\end{equation*}
For which
\begin{equation}\label{stima imm 1}
 \|\c_0(v)\|_{L^1(\partial B_r; H^{n-1})}\leq C\,r^{n-1}\,r^{-\frac{n}{p}} \|v\|_{W^{\l,p}(B_r)}.
\end{equation}

If $p\leq n$, let $\frac{n-(\l-t)p}{np}<t<\l$, i.e. $\frac{n-\l p}{(n-1)p}<t<\l$, of note that  $\l>\frac{n-\l p}{(n-1)p}$ if and only if $\l>\frac{1}{p}$.
We infer by Theorem \ref{cor imm fraz} and by Theorem \ref{toerema-traccia-fraz} the following
\begin{equation*}
  W^{\l,p}(B_r)\hookrightarrow W^{t,\frac{np}{n-(\l-t)p}}(B_r)\hookrightarrow W^{t-\frac{n-(\l-t)p}{np},\frac{np}{n-(\l-t)p}}(\partial B_r; \mathcal{H}^{n-1}) \hookrightarrow L^1(\partial B_r; \mathcal{H}^{n-1}).
 \end{equation*}
Applying the same reasoning to deduce \eqref{stima imm 1}, in particular we achieve
 \begin{equation}\label{stima imm q}
 \|\c_0(v)\|_{L^1(\partial B_r; H^{n-1})}\leq C\, r^{n-1}\,r^{-\frac{n-(\l-t)p}{p}} \|v\|_{W^{t,\frac{np}{n-(\l-t)p}}(B_r)}.
\end{equation}
\end{oss}

\section{The classical obstacle problem}
Let $\Omega\subset \R^n$ be a smooth, bounded and open set, $n\geq 2$, let $\mathbb{A}:\Omega\to \R^{n\times n}$ be a matrix-valued field 
and $f:\Omega\to \R$ be a function satisfying assumptions $(H1)$-$(H3)$ seen in the Introduction.

\begin{oss}
 By \eqref{A coerc cont}, we immediately deduce  that $\A$ is bounded. In particular, $\|\A\|_{L^\infty(\Om)}\leq \Lambda$.  
\end{oss} 

We define, for every open $A\subset \Omega$ and for each function $v\in H^1(\Omega)$, the following energy:
\begin{equation}
 \mathcal{E}[v,A]:=\int_A \left(\langle \A(x)\nabla v(x),\nabla v(x)\rangle +2f(x)u(x)\right)\,dx,
\end{equation}
with $\mathcal{E}[v,\Omega]:=\mathcal{E}[v]$.

\begin{prop}
We consider the following minimum problem with obstacle:  
\begin{equation}\label{problema ad ostacolo}
\inf_K \mathcal{E}[\cdot],
\end{equation}
where $K\subset H^1(\Omega)$ is the weakly closed convex given by
\begin{equation}
 K:=\{v\in H^1(\Omega)\, |\,  v\geq 0\, \mathcal{L}^n\textit{-a.e. on}\,\, \Omega,\, \c_0(v)=g\,\textit{on}\,\, \partial\Omega \},
\end{equation}
with $g\in H^\frac{1}{2}(\partial\Omega)$ being a nonnegative function.

Then there exists a unique solution for the minimum problem \eqref{problema ad ostacolo}.
\end{prop}
\begin{proof}
The hypotheses $(H1)$-$(H3)$ imply 
that the energy $\mathcal{E}$ is coercive and strictly convex in $K$. 
Thus, $\mathcal{E}$ is lower semicontinuous for the weak topology in $H^1(\Om)$, and so there exists a unique minimizer
that, as we stated in the introduction, we will indicate by $u$.
\end{proof}
Now, we can fix the notation for the \emph{coincidence set}, \emph{non-coincidence set} and the \emph{free-boundary} 
by defining the following:
\begin{equation}\label{notation free-boundary set}
 \Lambda_u:=\{u=0\},\qquad N_u:=\{u>0\},\qquad \Gamma_u=\partial\Lambda_u\cap\Om. 
\end{equation}

Actually, the minimum $u$ satisfies the partial differential equation both in the distributional sense and a.e. on $\Om$.
Therefore it shows good properties of regularity: 
\begin{prop}\label{prop PDE_u}
Let $u$ be the minimum of $\mathcal{E}$ in $K$. Then 
 \begin{equation}\label{PDE_u}
  \mathrm{div}(\A(x)\nabla u(x))=f(x)\chi_{\{u>0\}}(x) \qquad \textit{a.e. on $\Om$ and in $\mathcal{D}'(\Om)$}. 
 \end{equation}
Therefore,
\begin{itemize}
 \item[(i)] if $ps<n$, called  $p^*(s,p):=p^*=\frac{np}{n-sp}$, we have $u\in W^{2,p^*}\cap C^{1,1-\frac{n}{p^*}}(\Om)$;
 \item[(ii)] if $ps=n$ we have $u\in W^{2,q}\cap C^{1,1-\frac{n}{q}}(\Om)$ for all $1<q<\infty$.
\end{itemize}
\end{prop}

\begin{proof}
For the first part of proof we refer to \cite[Proposition 2.2]{FGS} and \cite[Proposition 3.2]{FGerS}.

Now, based on equation \eqref{PDE_u}, we can prove (i), the regularity of $u$ if $ps<n$.  
From Theorem \ref{cor imm fraz} $W^{1+s,p}(\Om) \hookrightarrow W^{1,p^*}(\Om)$ with $p^*=\frac{np}{n-sp}$. 
We also note that by the hypothesis $(H1)$ $p^*>n$, so by Morrey theorem $\A\in C^{0,1-\frac{n}{p^*}}(\Om)$.
 Since $u$ is the solution of \eqref{PDE_u}, and 
 thanks to \cite[Theorem 3.13]{HL}, $u\in C_{loc}^{1,1-\frac{n}{p^*}}(\Om)$. 
 We consider the equation
 \begin{align}\label{PDE u Miranda}
  \mathrm{Tr}(\A\nabla^2 v) = f\chi_{\{u>0\}}-\sum_{j}\mathrm{div}(a^j)\frac{\partial u}{\partial x_j}=:\varphi,
 \end{align}
 where the symbol $\mathrm{Tr}$ is the \emph{trace} of the matrix $\A\nabla^2 v$ and $a^j$ denotes the $j$-column of $\A$. 
 Since $\nabla u\in L^\infty_{loc}(\Om)$ and $\diw(a^j)\in L^{p^*}(\Om)$ for all $j\in\{1,\dots,n\}$ 
 then $\varphi\in L^{p^*}_{loc}(\Omega)$. 
 So, from \cite[Corollary 9.18]{GT} there exists a unique $v\in W_{loc}^{2,p^*}(\Om)$
 solution of \eqref{PDE u Miranda}. 
 We observe that the identity 
 $\mathrm{Tr}(\A\nabla^2 v)=\mathrm{div}(\A\nabla v)-\sum_{j}\mathrm{div}(a^j)\frac{\partial u}{\partial x_j}$ is verified.
 So, if we rewrite \eqref{PDE u Miranda} as follows
\begin{equation}
 \mathrm{div}(\A\nabla v)-\sum_{j}\mathrm{div}(a^j)\frac{\partial u}{\partial x_j} =\varphi,
\end{equation}
we have that $u$ and $v$ are two solutions. Then by \cite[Theorem 8.3]{GT} we obtain $u=v$ and the thesis follows.
Instead, if $ps=n$ from \cite[Theorem 6.10]{NPV}, 
$\A\in W^{1,q}$ and so $u\in W^{2,q}\cap C^{1,1-\frac{n}{q}}(\Om)$ for all $1<q<\infty$. 
Applying the same reasoning to deduce the item (i) we obtain the item (ii) of the thesis.
\end{proof}
We note that thanks to continuity of $u$ the sets defined in \eqref{notation free-boundary set} 
are pointwise defined and we can equivalently write $\Gamma_u=\partial N_u\cap\Om$.
\begin{oss}
 The assumption $f\geq c_0>0$ in $(H3)$ is not necessary in order to prove the regularity of $u$ and that the minimum $u$ satisfies 
 the equation~\eqref{PDE_u} (cf. \cite[Proposition~3.2, Theorem~3.4, Corollary~3.5]{FGerS}).
\end{oss}

\section{The blow-up method: Existence of blow-ups and nondegeneracy of the solution}
In this section we shall investigate the existence of blow-ups. 
In this connection, we need to introduce for any point $x_0\in \Gamma_u$ a sequence of rescaled functions: 
\begin{equation}\label{u_x_0 r}
u_{x_0,r}:=\frac{u(x_0+rx)}{r^2}.
\end{equation}
We want to prove the existence of limits (in a strong sense) of this sequence as $r\to 0^+$ and define these \emph{blow-ups}.
We start observing that the rescaled function satisfies an appropriate PDE and satisfies uniform $W^{2,p^*}$ estimates. 
We can prove this thanks to the regularity theory for elliptic equations.

\begin{prop}\label{prop u_r limitata W2p}
 Let $u$ be the solution to the obstacle problem \eqref{problema ad ostacolo} 
and $x_0\in \Gamma_u$.
 Then, for every $R>0$
 there exists a constant $C>0$ such that, for every $r\in (0,\frac{\mathrm{dist}(x_0, \partial\Om)}{4R})$
 \begin{equation}\label{u_r limitata W2p}
  \|u_{x_0,r}\|_{W^{2,p^*}(B_R(x_0))}\leq C.
 \end{equation}
In particular, the functions $u_{x_0,r}$ are equibounded in $C^{1,\c'}$ for $\c'\leq \c:=1-\frac{n}{p^*}$.
\end{prop}
\begin{proof}
 From \eqref{u_x_0 r} and Proposition \ref{prop PDE_u} it holds
 \begin{equation}\label{PDE_ur}
  \mathrm{div}(\A(x_0+rx)\nabla u_{x_0,r}(x))=f(x_0+rx)\chi_{\{u_{x_0,r}>0\}}(x) \quad \textit{a.e. on $B_{4R}(x_0)$ and on $\mathcal{D}'(B_{4R}(x_0))$},
 \end{equation}
 and $u_{x_0,r}\in W^{2,p^*}\cap C^{1,\c}(B_{4R}(x_0))$. 
 We have $x_0\in \Gamma_u$, then $u_{x_0,r}(\underline{0})=0$.
 Since $u_{x_0,r}\geq 0$, from \cite[Theorems $8.17$ and $8.18$]{GT} 
 we have 
\begin{equation}\label{u_r leq f}
 \|u_{x_0,r}\|_{L^\infty(B_{4R}(x_0))}\leq C(R,x_0)\|f\|_{L^\infty(B_{4R}(x_0))}. 
\end{equation}
Thanks to \cite[Theorem 8.32]{GT} and \eqref{u_r leq f} we obtain
\begin{equation}\label{u_r C1 hold <f}
 \|u_{x_0,r}\|_{C^{1,\c}(B_{2R}(x_0))}\leq C\, \big(\|u_{x_0,r}\|_{L^\infty(B_{4R}(x_0))} + \|f\|_{L^\infty(B_{4R}(x_0))}\big)\leq C' \|f\|_{L^\infty(B_{4R}(x_0))}.
\end{equation}
We observe that, as in Proposition \ref{prop PDE_u}, $u_{x_0,r}$ is solution to 
\begin{equation}\label{PDE u_r Miranda}
 \mathrm{Tr}\left(\A(x_0+rx)\nabla^2u_{x_0,r}(x)\right) = f(x_0+rx)\chi_{\{u_{x_0,r}>0\}} -  r \sum_{j}\diw\left(a^j(x_0+rx)\right)\frac{\partial u_{x_0,r}}{\partial x_j}(x)=:\varphi_r(x),
\end{equation}
with $\varphi_r\in L^{p^*}(B_{2R}(x_0))$. Then from \cite[Theorem 9.11]{GT}
\begin{equation}
 \|u_{x_0,r}\|_{W^{2,p^*}\left(B_R(x_0)\right)}\leq C \left(\|u_{x_0,r}\|_{L^{p^*}(B_{2R}(x_0))} + \|\varphi_r\|_{L^{p^*}(B_{2R}(x_0))}\right).
\end{equation}
We define $\diw(\A):=(\diw(a^j))_j$, namely the vector of divergence of the vector column of $\A$. Then by \eqref{u_r C1 hold <f}
\begin{equation*}
 \begin{split}
  \|\varphi_r\|^{p^*}_{L^{p^*}(B_{2R}(x_0))}&=\int_{B_{2R}(x_0)} |f(rx)\chi_{\{u_r>0\}} -  r\langle\diw\A(rx),\nabla u_r(x)\rangle|^{p^*}\,dx\\
  &\leq C\,\|f\|^{p^*}_{L^\infty(B_{4R}(x_0))}\bigg(1 + r^{p^*-n}\int_{B_{2rR}(x_0)} |\langle\diw\A(y)|^{p^*}\,dy\bigg)\\
  &\leq C\,\|f\|^{p^*}_{L^\infty(B_{4R}(x_0))}\bigg(1 + \Big(\frac{\mathrm{dist}\left(x_0, \partial\Om\right)}{4R}\Big)^{p^*-n}\|\diw\A(y)\|^{p^*}_{W^{1,p^*}(\Om)}\bigg).
 \end{split}
\end{equation*}
So $\|u_{x_0,r}\|_{W^{2,p^*}}(B_R(x_0))\leq C$, where $C$ does not depend on $r$.
\end{proof}

\begin{cor}[Existence of blow-ups]
 Let $x_0\in \Gamma_u$ with $u$ the solution of \eqref{problema ad ostacolo}. 
Then for every sequence $r_k\downarrow 0$ there exists a subsequence $(r_{k_j})_j\subset (r_k)_k$ such that the rescaled functions 
$(u_{x_0,r_{k_j}})_j$ converge in $C^{1,\gamma}$.
 We define these limits as \emph{blow-ups}.
\end{cor}
\begin{proof}
The proof is an easy consequence of Proposition \ref{prop u_r limitata W2p} and the Ascoli-Arzelà Theorem.
\end{proof}

\begin{oss}
Recalling $x_0\in \Gamma_u$ we have $u(x_0)=0$ and $\nabla u(x_0)=0$ so 
\begin{equation}\label{u e grad u limitate}
 \|u\|_{L^\infty(B_r)(x_0)}\leq C\, r^2 \qquad \mathrm{and} \qquad \|\nabla u\|_{L^\infty(B_r(x_0))}\leq C\, r.   
\end{equation}
We note that the constant in \eqref{u e grad u limitate} only depends on the constant $C$ in \eqref{u_r limitata W2p} 
and is therefore uniformly bounded for points $x_0\in \Gamma_u\cap K$ for each compact set $K\subset \Om$.
\end{oss}

As in classical case, the solution $u$ has a quadratic growth. 
The lacking regularity of the problem does not allow us to use the classic approach by Caffarelli \cite{Caf98} also used by
Focardi, Gelli and Spadaro in \cite[Lemma 4.3]{FGS}. 
The main problem is that $\diw(a^j)$, that is a $W^{1,p^*}$ function, 
is not a priori pointwise defined, so the classical argument fails. 
We use a more general result of Blank and Hao in \cite[Chapter 3]{Blank}.
\begin{prop}[{\cite[Theorem 3.9]{Blank}}]\label{prop crescita basso}
 Let $x_0\in \Gamma_u$, and $u$ be the minimum of \eqref{problema ad ostacolo}. Then, there exists a constant $\theta>0$ such that
 \begin{equation}
  \sup_{\partial B_r(x_0)} u \geq \theta\,r^2.
 \end{equation}
\end{prop}

To proceed in the analysis of the blow-ups we shall prove a monotonicity formula. This will be a key ingredient to prove the $2$-homogeneity
of blow-ups and that blow-ups are nonzero. Therefore it allows us to classify blow-ups. 
This result will be the focus of Section \ref{s:classification}, while the quasi-monotonicity formula will be 
the topic of Section \ref{s:Weiss}.

\section{Weiss’ quasi-monotonicity formula}\label{s:Weiss}
In this section we show that the monotonicity formulae established by Weiss \cite{Weiss} and Monneau \cite{Monneau} 
in the Laplace case ($\A\equiv I_n$) and by Focardi, Gelli and Spadaro \cite{FGS} in the $\A$ Lipschitz continuous 
and $f$ H\"{o}lder continuous case, hold in our case as well.

As in \cite{FGS} we proceed by fixing the coordinates system: let $x_0\in \Gamma_u$, be any point of free-boundary, then the affine change of variables
\begin{equation}
 x \longmapsto x_0 + f(x_0)^{-\frac{1}{2}}\A^{\frac{1}{2}}(x_0)x = x_0 + \LL(x_0)x
\end{equation}
leads to
\begin{equation*}
 \mathcal{E}[u,\Om]=f^{1-\frac{n}{2}}(x_0)\, \textrm{det}(\A^{\frac{1}{2}}(x_0))\,\, \mathcal{E}_{\LL(x_0)}[u_{\LL(x_0)},\Om_{\LL(x_0)}],
\end{equation*}
with the following notations:
\begin{equation}\label{notazioni trasformate}
 \begin{split}
 \mathcal{E}_{\LL(x_0)}[v,A]&:=\int_A \left(\langle \C_{x_0}\nabla v, \nabla v \rangle + 2\frac{f_{\LL(x_0)}}{f(x_0)}v\right)\,dx \qquad \forall\,A\subset \Om_{\LL(x_0)},\\
 \Om_{\LL(x_0)}&:=\LL(x_0)^{-1}(\Om - x_0),\\
 u_{\LL(x_0)}(x)&:=u(x_0+\LL(x_0)x),\\
 f_{\LL(x_0)}&:=f(x_0+\LL(x_0)x),\\
 \C_{x_0}&:=\A^{-\frac{1}{2}}(x_0) \A(x_0+\LL(x_0)x) \A^{-\frac{1}{2}}(x_0),\\
 u_{\LL(x_0),r}(y)&:=\frac{u(x_0+r\LL(x_0)y)}{r^2}.
 \end{split}
\end{equation}
We observe that the image of the free-boundary in the new coordinates is:
\begin{equation}
 \Gamma_{u_{\LL(x_0)}}=\LL(x_0)^{-1}(\Gamma_u - x_0)
\end{equation}
and we see how energy $\mathcal{E}$ is minimized by $u$, if and only if, the energy $\mathcal{E}_{\LL(x_0)}$ is minimized by $u_{\LL(x_0)}$.

Therefore, for a fixed base point $x_0\in \Gamma_u$, we change the coordinates system and as we stated before
\begin{equation*}
 \underline{0}\in \Gamma_{u_{\LL(x_0)}} \qquad\qquad \C_{x_0}(\underline{0})=I_n\qquad\qquad f_{\LL(x_0)}(\underline{0})=f(x_0). 
\end{equation*}
The point of the choice of this change of variable is that, in a neighborhood of $\underline{0}$, 
the functional $\mathcal{E}_{\LL(x_0)}[v,\Om]$ is a perturbation of $\int_{\Om}(|\nabla v|^2 +2v)\,dx$, 
which is the functional associated with the classical Laplacian case.
We identify the two spaces in this section to simplify the ensuing calculations,
then with a slight abuse of notation we reduce to \eqref{abuso_notazione}:
\begin{equation*}
 x_0=\underline{0}\in \Gamma_u,\qquad \A(\underline{0})=I_n,\qquad f(\underline{0})=1.
\end{equation*}
We note that with this convention $\underline{0}\in\Om$. 
In the new coordinates system we define 
\begin{equation}\label{def mu}
 \nu(x):=\frac{x}{|x|}\quad \textit{for}\,\, x\neq\underline{0}, \qquad\qquad\qquad  
	    \mu(x):=\left\{  \begin{array}{ll}
			  \langle \A(x)\nu(x), \nu(x)\rangle 		&	\quad \textrm{if} \,\,	x\neq\underline{0}\\
			  1						& 	\quad \textrm{otherwise}.
			\end{array} \right.
\end{equation}
We note that $\mu\in C^0(\Om)$ by $(H1)$ and \eqref{abuso_notazione}. We prove the following result:

\begin{lem}\label{lem reg mu}
 Let $\A$ be a matrix-valued field. Assume that $(H1)$, $(H2)$ and \eqref{abuso_notazione} hold, then
 \begin{equation}
  \mu\in W^{1,q}\cap C^{0,1-\frac{n}{p^*}}(\Om) \qquad\qquad \forall q<p^*,
 \end{equation}
and
\begin{equation}\label{mu limitata}
 \Lambda^{-1}\leq \mu(x)\leq \Lambda \qquad\qquad \forall x\in\Om.
\end{equation}
\end{lem}
\begin{proof}
 We prove that $\mu\in W^{1,q}$ for any $q<p^*$.
 
 We use a characterization of Soblev's spaces (see \cite[Proposition IX.3]{Brezis}): 
 $\mu \in W^{1,q}(\Om)$ if and only if there exists  a constant $C>0$ such that for every open $\omega\subset\subset\Om$ and 
for any $h\in \R^n$ with $|h|<\mathrm{dist}(\omega, \partial\Om)$ it holds
 \begin{equation*}
  \|\tau_h \mu - \mu\|_{L^q(\omega)}\leq C\,|h|.
 \end{equation*}
For the convexity of the function $|\cdot|^q$, remembering that $\A$ is $(1-\frac{n}{p^*})$-H\"{o}lder continuous 
and since for all $x\in R^n$ and $h\neq -x$ the inequality $\left|(x+h)\frac{|x|}{|x+h|} - x \right|\leq 2|h|$ holds, we have
\begin{equation*}
 \begin{split}
\|\tau_h \mu - &\mu\|_{L^q(\omega)}^q = \int_\omega \left|\langle\A(x+h)\frac{x+h}{|x+h|},\frac{x+h}{|x+h|}\rangle - \langle\A(x)\frac{x}{|x|},\frac{x}{|x|}\rangle\right|^q\,dx\\ 
&=\int_\omega \bigg|\langle(\A(x+h)-\A(x))\frac{x+h}{|x+h|},\frac{x+h}{|x+h|}\rangle + \langle\left(\A(x)-\A(\underline{0})\right)\left(\frac{x+h}{|x+h|}+\frac{x}{|x|}\right),\left(\frac{x+h}{|x+h|}-\frac{x}{|x|}\right)\rangle\bigg|^q\,dx\\
&\leq 2^{q-1}\bigg(\int_\omega |\A(x+h)-\A(x)|^q\,dx + \int_\omega 2^q \bigg|\frac{\A(x)-\A(\underline{0})}{|x|}\bigg|^q\,\,\bigg| (x+h)\frac{|x|}{|x+h|} - x \bigg|^q\ \bigg)\\
&\leq 2^{q-1}\bigg(c |h|^q + 4^q |h|^q \int_\omega \frac{1}{|x|^{n\frac{q}{p^*}}}\,dx\bigg) = C\,|h|^q,  
\end{split}
\end{equation*}
where in the last equality, we rely on $|x|^{-\frac{nq}{p^*}}$ being integrable if and only if $q<p^*$.
By Sobolev embedding theorem, we have $\mu\in C^{0,1-\frac{n}{q}}$ for any $q<p^*$. 

Thanks to the structure of $\mu$ 
we can earn more regularity. In particular $\mu\in C^{0,\c}$ with $\c=1-\frac{n}{p^*}$.
We start off proving the inequality when one of the two points is $\underline{0}$:
\begin{align*}
 |\mu(x)-\mu(\underline{0})|&=\bigg|\langle \A(x)\frac{x}{|x|}, \frac{x}{|x|}\rangle - 1\bigg| = \bigg|\langle\A(x)\frac{x}{|x|}, \frac{x}{|x|}\rangle - \langle\frac{x}{|x|}, \frac{x}{|x|}\rangle\bigg|\\
		& =\bigg|\langle (\A(x)-\A(\underline{0}))\frac{x}{|x|}, \frac{x}{|x|}\rangle\bigg| = [A]_{C^{0,\c}} \, |x|^\c.
\end{align*}
Let us assume now that $x,y\neq \underline{0}$ and prove the inequality in the remaining case. Let $z=|y|\frac{x}{|x|}$ then  
\begin{equation*}
 |\mu(x)-\mu(y)|\leq |\mu(x)-\mu(z)| + |\mu(z)-\mu(y)|.
\end{equation*}
As $\frac{z}{|z|}=\frac{x}{|x|}$ 
\begin{equation*}
 |\mu(x)-\mu(z)|=\bigg|\langle (\A(x)-\A(z))\frac{x}{|x|}, \frac{x}{|x|}\rangle\bigg| \leq [A]_{C^{0,\c}} \, |x-z|^\c,
\end{equation*}
while by $|z|=|y|=r$
\begin{equation*}
\begin{split}
 |\mu(z)&-\mu(y)|=\bigg|\langle\A(z)\frac{z}{r},\frac{z}{r}\rangle - \langle\A(y)\frac{y}{r},\frac{y}{r}\rangle\bigg| 
 \leq \bigg|\langle(\A(z)-\A(y))\frac{z}{r},\frac{z}{r}\rangle\bigg| + \bigg|\langle\A(y))\frac{z}{r},\frac{z}{r}\rangle - \langle\A(y)\frac{y}{r},\frac{y}{r}\rangle\bigg| \\
 &\leq [\A]_{C^{0,\c}}|z-y|^\c + \bigg|\langle (\A(y)-\A(0))\frac{z+y}{r},\frac{z-y}{r}\rangle\bigg|
 \leq [\A]_{C^{0,\c}}\bigg(|z-y|^\c + 2 \frac{|z-y|^{1-\c}}{r^{1-\c}} |z-y|^\c \bigg)\\
 &\leq [\A]_{C^{0,\c}}\bigg(|z-y|^\c + 2^{1-\c}|z-y|^\c \bigg)\leq C [\A]_{C^{0,\c}}|z-y|^\c.
 \end{split}
\end{equation*}
Therefore, since $|x-z|=| |x|-|y||\leq |x-y|$ and $|z-y|\leq |z-x|+|x-y|\leq 2|x-y|$
we have the thesis
\begin{equation*}
|\mu(x)-\mu(y)|\leq C [\A]_{C^{0,\c}}|x-y|^\c.
\end{equation*}
\end{proof}
We introduce rescaled volume and boundary energies
\begin{equation}
\begin{split}\label{Er}
 \mathcal{E}(r)&:=\mathcal{E}[u,B_r]=\int_{B_r}\left(\langle \A(x)\nabla u(x),\nabla u(x)\rangle +2f(x)u(x)\right)\,dx \\
 &=r^{n+2}\int_{B_1}\left(\langle \A(rx)\nabla u_r(x),\nabla u_r(x)\rangle +2f(rx)u_r(x)\right)\,dx
\end{split}
\end{equation}
\begin{equation}\label{Hr}
  \mathscr{H}(r):=\int_{\partial B_r}\mu(x)u^2(x)\,d\mathcal{H}^{n-1}=r^{n+3}\int_{\partial B_1}\mu(rx)u_r^2(x)\,d\mathcal{H}^{n-1}.
\end{equation}
We now introduce an energy ``à la Weiss'' combining and rescaling the terms above:
\begin{equation}\label{Phir}
 \Phi(r):=r^{-n-2}\mathcal{E}(r) - 2\,r^{-n-3}\mathscr{H}(r).
\end{equation}

\begin{oss}
 By \eqref{abuso_notazione}, \eqref{Er}, \eqref{Hr} and Proposition \ref{prop u_r limitata W2p} we have
 \begin{equation}\label{E H-O_grande}
  \begin{split}
  \mathcal{E}(r)&=\int_{B_r}(|\nabla u_r|^2 + 2u)\,dx + O(r^{n+2+\min(\c,\a)})\stackrel{\eqref{u e grad u limitate}}{=} O(r^{n+2}),  \\
  \mathscr{H}(r)&=\int_{\partial B_r}u^2\,d\mathcal{H}^{n-1} + O(r^{n+3+\c})\stackrel{\eqref{u e grad u limitate}}{=} O(r^{n+3}).  
  \end{split}
 \end{equation}
 Hence, the choice of the renormalized factors in \eqref{Phir}.
\end{oss}

To complete the notation in \eqref{notazioni trasformate} we show the trasformed version of \eqref{def mu} and \eqref{Phir}: 
\begin{equation}\label{notazioni LL e energie}
 \begin{split}
  \mu_{\LL(x_0)}(y):=&\langle \C_{x_0}(y)\nu(y),\nu(y)\rangle \qquad y\neq \underline{0}, \qquad\qquad \mu_{\LL(x_0)}(\underline{0}):=1,\\
\Phi_{\LL(x_0)}(r):=&\int_{B_1}\big(\langle\C_{x_0}(ry)\nabla u_{\LL(x_0),r}(y), \nabla u_{\LL(x_0),r}(y) \rangle
				    +2\,\frac{f_{\LL(x_0)}(ry)}{f(x_0)}u_{\LL(x_0),r}\big)\,dy\\
				   & -2\int_{\partial B_1}\mu_{\LL(x_0)}(ry)\,u_{\LL(x_0),r}^2(y)\,d\mathcal{H}^{n-1}
\end{split}
\end{equation}
\begin{oss}
We can note by the definition above and in view of Lemma~\ref{lem reg mu} $\Lambda^{-2}\leq \mu_{\LL(x_0)}(y)\leq \Lambda^2$ 
and $\mu_{\LL(x_0)}\in~C^{0,\c}(\Om)$.
\end{oss}
\subsection{Estimate of derivatives of $\mathcal{E}$ and $\mathscr{H}$}
To estimate the derivative of auxiliary energy $\Phi$ we estimate the derivative of addenda $\mathcal{E}$ and $\mathscr{H}$. 
Starting with $\mathcal{E}$, for this purpose, following Focardi, Gelli and Spadaro \cite{FGS},
we use a  generalization of Rellich–Necas’ identity due to Payne–Weinberger \cite[Lemma 3.4]{FGS} in order to calculate 
the derivative.

\begin{prop}\label{E'}
 There exists a constant $C_1>0$, $C_1=C_1(\l,C,\|\A\|_{W^{1+s,p}(\Om)})$, 
such that for $\mathcal{L}^1$-a.e. $r\in(0,\mathrm{dist}(\underline{0},\partial\Om))$, 
 \begin{equation}\label{estimate E'}
  \begin{split}
   \mathcal{E}'(r)&=2\int_{\partial B_r}\mu^{-1}\langle \A\nu, \nabla u \rangle^2\,d\mathcal{H}^{n-1} + \frac{1}{r}\int_{B_r}\langle\A\nabla u,\nabla u\rangle\,\diw(\mu^{-1}\A x)\,dx\\
   &-\frac{2}{r}\int_{B_r}f \langle\mu^{-1}\A x,\nabla u\rangle\,dx - \frac{2}{r}\int_{B_r}\langle\A\nabla u, \nabla^T(\mu^{-1}\A x)\nabla u\rangle\, dx\\
   &+2\int_{\partial B_r} f u\, d\mathcal{H}^{n-1} + \eps(r), 
  \end{split}
 \end{equation}
with $|\eps(r)|\leq C_1 \mathcal{E}(r)\,r^{-\frac{n}{p^*}}$. 
\end{prop}
\begin{proof}
For details of proof we refer to \cite[Proposition 3.5]{FGS}.
The difference consists in the different regularity of $\diw\A$ that in our case is $L^{p^*}$ instead of $L^\infty$.
In the estimate of $\varepsilon(r)$, made through H\"{o}lder inequality, the factor $r^{-\frac{n}{p^*}}$ appears. 
This is unbounded but at the same time it is integrable in $r$ and this will be crucial in proving the quasi-monotonicity formula.    
\end{proof}

\begin{oss}
The term $\eps(r)$ is exactly:
 \begin{equation*}
 \eps(r)=\frac{1}{r}\int_{B_r}\mu^{-1}\, (\nabla\A : \A x \otimes \nabla u\otimes \nabla u)\,dx.
\end{equation*}
In \cite[Proposition 3.5]{FGS} under the hypothesis $\A\in W^{1,\infty}(\Om,\R^n\times~\R^n)$, 
Focardi, Gelli and Spadaro showed the equality \eqref{estimate E'} with $\eps(r)$ as above, 
and proved that $|\eps(r)|\leq C\, \mathcal{E}(r)$.
\end{oss}

The next step is to estimate the derivative of $\mathscr{H}(r)$. By definition $\mathscr{H}(r)$ is a boundary integral; 
we follow the strategy of \cite[Proposition 3.6]{FGS} which consists in bringing us back to a volume 
integral using the Divergence Theorem and deriving through the Coarea formula. 
The difficulty is that we have to integrate  the function 
$\diw\A$ on $\partial B_r$, 
but by $(H1)$ $\diw\A$ is a function in $W^{s,p}(\Om)$ 
with $s>\frac{1}{p}$, and it is not, a priori, well defined on $\partial B_r$. 
Then, taking into account the concept of \emph{trace} we can prove a corollary of the  Coarea formula.

\begin{prop}\label{coarea-traccia}
 Let $\varphi\in W^{\l,p}(B_1)$ with $\l>\frac{1}{p}$. Then, denoting by $\c_0(\phi)$ the trace of $\phi$ (see Theorem \ref{toerema-traccia-fraz}),
 for $\mathcal{L}^1$-a.e. $r\in (0,1)$ it holds
 \begin{equation}
  \frac{d}{dr}\bigg(\int_{B_r} \varphi\,dx\bigg) = \int_{\partial B_r} \c_0(\varphi)\,d\mathcal{H}^{n-1}.
 \end{equation}
 \end{prop}
\begin{proof}
Let $(\varphi_j)_j\subset C^\infty(\overline{B_1})$ such that $\varphi_j\to \varphi$ in $W^{\l,p}(B_1)$. 
 For each function $g_j$, by the Coarea formula for $\mathcal{L}^1$-a.e. $r\in (0,1)$ it holds that
 \begin{equation}\label{coarea_j}
\frac{d}{dr}\bigg(\int_{B_r} \varphi_j\,dx\bigg) = \int_{\partial B_r} \varphi_j\,d\mathcal{H}^{n-1}.
 \end{equation}
 By the continuity of trace and Lebesgue's dominated convergence Theorem we have
 \begin{equation}\label{cont-j-traccia}
  \lim_j \int_{\partial B_r} \varphi_j\,d\mathcal{H}^{n-1} = \int_{\partial B_r} \c_0(\varphi)\,d\mathcal{H}^{n-1}.
 \end{equation}
Let us now prove that $\lim_j\frac{d}{dr}\bigg(\int_{B_r} \varphi_j\,dx\bigg)=\frac{d}{dr}\bigg(\int_{B_r} \varphi\,dx\bigg)$.

In this connection we define the function $G(r):=\int_{B_r} \varphi\,dx$ and the sequence $G_j(r):=\int_{B_r} \varphi_j\,dx$; 
we can prove that $G_j\to G$ in $W^{1,1}((0,1))$.

We recall that by a well-known characterization, the functions in $W^{1,1}$ on an interval are absolutely continuous functions.
In order to deduce that $G,G_j\in W^{1,1}$ we have to prove that 
for any $\eps>0$ there exists a $\d>0$ such that for any finite sequence of disjoint intervals
 $(a_k,b_k)\subset (0,1)$ the condition $\sum_k |G_j(b_k)-G_j(a_k)|<\eps$ holds if $\sum_k |b_k-a_k|<\d$.
 
 Therefore, we estimate as follows
\begin{align*}
 \sum_k |G_j(b_k)-G_j(a_k)|=\sum_k \bigg|\int_{B_{b_k}\setminus B_{a_k}} \varphi_j\,dx\bigg|\leq \int_{\bigcup_k (B_{b_k}\setminus B_{a_k})} |\varphi_j|\,dx<\eps,
\end{align*}
where in the last inequality, we use the absolute continuity of the integral and
\begin{equation*}
 \mathcal{L}^n(\cup_k (B_{b_k}\setminus B_{a_k})=n\omega_n\int_{\cup_k (b_k,a_k)} r^{n-1}\,dr\leq n\omega_n\int_{1-\d}^1 r^{n-1}\,dr \leq n\omega_n \d.
\end{equation*}
The previous argument holds for $G$ as well. Thus $G$ and $G_j$ are differentiable $\mathcal{L}^1$-a.e. on~$(0,1)$. 
On the other hand by the Coarea formula, we can represent the weak derivative 
of $G_j$ in the following way:
\begin{equation*}
 G_j'(r)=\int_{\partial B_r} \varphi_j\,d\mathcal{H}^{n-1}, \quad G'(r)=\int_{\partial B_r} \varphi\,d\mathcal{H}^{n-1} \qquad\qquad \mathcal{L}^1\textit{-a.e.}\,\,r\in(0,1).
\end{equation*}
Thus
\begin{align*}
 \|G_j-G\|_{L^1((0,1))}&=\int_0^1 \Big|\int_{B_r}(\varphi_j-\varphi)\,dx\Big|\,dr\leq \|\varphi_j-\varphi\|_{L^1(B_1)}\xrightarrow{j\to \infty} 0,\\
 \|G_j'-G'\|_{L^1((0,1))}&=\int_0^1 \Big|\int_{\partial B_r}(\varphi_j-\varphi)\,d\mathcal{H}^{n-1}\Big|\,dr\leq \|\varphi_j-\varphi\|_{L^1(B_1)}\xrightarrow{j\to \infty} 0,
\end{align*}
therefore up to subsequence $G_j'\to G'$ $\mathcal{L}^1$-a.e. $(0,1)$; then by combining together this, \eqref{coarea_j} and 
\eqref{cont-j-traccia} we have the thesis. 
\end{proof}

We estimate the derivative of $\mathscr{H}(r)$.\\
We define an exponent $\Theta=\Theta(s,p,n,t_0)$, with $t_0\in \left(\frac{n-sp}{p(n-1)}, s\right)$ as in Remark \ref{oss stime imm e traccia},
for which the term $r^{-\frac{n}{\Theta}}$ is integrable.
For this purpose we define:
\begin{equation}\label{B}
		 \Theta=\Theta(s,p,n,t_0)=\left\{  \begin{array}{ll}
			  p &		 		   		\quad \textrm{if} \,\,	p>n\\
			  \frac{np}{n-(s-t_0)p} &         		\quad \textrm{if} \,\,	p\leq n.
			\end{array} \right.
\end{equation}
\begin{oss}\label{choice (H1)}
If $p>n$ the condition is trivial. If instead $p\leq n$ the condition $\frac{np}{n-(s-t_0)p}>n$
is equivalent to $t_0 < s+1-\frac{n}{p}$. Now such a $t_0$ exists if and only if $\frac{n-sp}{p(n-1)}<s + 1-\frac{n}{p}$
which is equivalent to requiring that $p>\frac{n^2}{n(1+s)-1}$. This explains the choice of the condition $(H1)$.
\end{oss}

\begin{prop}\label{H'}
There exists a positive constant $C_2=C_2(\|\A\|_{W^{1+s,p}})$ such that 
for $\mathcal{L}^1$-a.e. $r\in(0,\mathrm{dist}(\underline{0},\partial \Om))$ the following holds
\begin{equation}\label{estimate H'}
  \mathscr{H}'(r) = \frac{n-1}{r}\mathscr{H}(r) + 2\int_{\partial B_r} u \langle \A\nu,\nabla u \rangle\, d\mathcal{H}^{n-1} + h(r),
 \end{equation}
with $|h(r)|\leq C_2 \mathscr{H}(r)r^{-\frac{n}{\Theta}}$, where the constant $\Theta$ is given in \eqref{B}.
\end{prop}
\begin{proof}
 From the Divergence Theorem we write $\mathscr{H}(r)$ as volume integral
 \begin{align*}
  \mathscr{H}(r)& = \frac{1}{r}\int_{\partial B_r} u^2(x)\langle\A(x)x,\nu \rangle\,d\mathcal{H}^{n-1}= \frac{1}{r}\int_{B_r} \diw\big(u^2(x)\A(x)x\big)\,dx \\
  &=\frac{2}{r}\int_{B_r}u\nabla u\cdot \A(x)x\,dx + \frac{1}{r} \int_{B_r}u^2(x)\mathrm{Tr}\A\,dx + \frac{1}{r} \int_{B_r}u^2(x)\diw\A(x)\cdot x\,dx.
 \end{align*}
By taking Coarea formula and Proposition \ref{coarea-traccia} into account, we have
\begin{align*}
\mathscr{H}'(r)= &-\frac{1}{r}\mathscr{H}(r) +2\int_{\partial B_r}\!\!u \langle \A\nu, \nabla u\rangle \,d\mathcal{H}^{n-1}
 + \frac{1}{r} \int_{\partial B_r}\!\!u^2\,\mathrm{Tr}\A \,d\mathcal{H}^{n-1} + \frac{1}{r} \int_{\partial B_r}\!\!u^2\c_0\big(\diw\A(x)\big)\cdot x\,d\mathcal{H}^{n-1}\\
=&\,\frac{n-1}{r}\mathscr{H}(r) + 2\int_{\partial B_r}\mu \langle \A\nu,\nabla u \rangle\, d\mathcal{H}^{n-1} + h(r),
\end{align*}
with 
\begin{equation}\label{hr}
h(r)=\frac{1}{r} \int_{\partial B_r}u^2\,\big(\mathrm{Tr}\A - n\mu\big) \,d\mathcal{H}^{n-1} + \frac{1}{r} \int_{\partial B_r}u^2\c_0\big(\diw\A(x)\big)\cdot x\,d\mathcal{H}^{n-1}=:I+II.     
\end{equation}
We estimate separately the two terms.

For the first term let us recall that the H\"{o}lder continuity of $\A$ and $\mu$, the condition \eqref{u e grad u limitate}
and the fact that $\A(\underline{0})=I_n$ and $\mu(\underline{0})=1$ hold, we have:
\begin{equation}\label{1add-H'}
\begin{split}
 |I|\, &=\frac{1}{r} \bigg|\int_{\partial B_r}u^2\,\sum_i \big(a_{ii}(x)- \mu(x)\big) \,d\mathcal{H}^{n-1}\bigg|
 \leq\frac{1}{r} \int_{\partial B_r}\!\!\!\!u^2\,\sum_i \big(|a_{ii}(x)-a_{ii}(\underline{0})|+ |\mu(\underline{0}) - \mu(x)|\big) \,d\mathcal{H}^{n-1}\\
 &\leq C'r^{n+3-\frac{n}{p^*}} \leq C'\mathscr{H}(r)\, r^{-\frac{n}{p^*}},
\end{split}
\end{equation}
where in the last inequality we use \eqref{E H-O_grande}.

For the second term from the H\"{o}lder inequality, by \eqref{u e grad u limitate}
and recalling Remark \ref{oss stime imm e traccia} according to which $\c_0(\diw\A)\in L^1(\partial B_r,\R^n;\mathcal{H}^{n-1})$ we have: 
\begin{equation}\label{2add-H'}
\begin{split}
 |II|\,&\leq \frac{1}{r} \int_{\partial B_r} u^2\big|\c_0\big(\diw\A\big)(x)\big|\,|x|\,d\mathcal{H}^{n-1}
 \leq C'r^4\||\c_0(\diw\A)|\|_{L^1(\partial B_r,\R^n;\mathcal{H}^{n-1})}. 
\end{split}
\end{equation}
We can now analyze separately the two cases $p>n$ and $p\leq n$.

We start with the case $p>n$. We use \eqref{stima imm 1}, \eqref{E H-O_grande} in \eqref{2add-H'} to obtain  
\begin{equation}\label{2add H' p}
\begin{split}
 |II|\,\leq C\, \|\diw\A\|_{W^{s,p}(B_r,\R^n;\mathcal{H}^{n-1})}\,r^{n+3}\, r^{-\frac{n}{p}} \leq C\,\|\diw\A\|_{W^{p,s}(\Omega,\R^n;\mathcal{H}^{n-1})}\,\mathscr{H}(r)\,r^{-\frac{n}{p}}
 \leq C\,\mathscr{H}(r)\,r^{-\frac{n}{p}}.
 \end{split}
 \end{equation}

If $p\geq n$ by \eqref{stima imm q} we have 
 \begin{equation*}
\|\c_0(\diw\A)\|_{L^1(\partial B_r,\R^n; H^{n-1})}\leq C\, r^{n-1}\,r^{-\frac{n-(s-t_0)p}{p}} \|\diw\A\|_{W^{t_0,\frac{np}{n-(s-t_0)p}}(B_r,\R^n;\mathcal{H}^{n-1})}.
 \end{equation*}
Hence, recalling \eqref{2add-H'} and \eqref{E H-O_grande}
\begin{equation}\label{2add H' q}
\begin{split}
 |II|\,&\leq C\, \|\diw\A\|_{W^{t_0,\frac{np}{n-(s-t_0)p}}(B_r,\R^n;\mathcal{H}^{n-1})}\,r^{n+3}\, r^{-\frac{n-(s-t_0)p}{p}}\\
 &\leq C\,\|\diw\A\|_{W^{t_0,\frac{np}{n-(s-t_0)p}}(\Omega,\R^n;\mathcal{H}^{n-1})}\,\mathscr{H}(r)\,r^{-\frac{n-(s-t_0)p}{p}}\\
 &\leq C\,\|\diw\A\|_{W^{s,p}(\Omega,\R^n;\mathcal{H}^{n-1})}\,\mathscr{H}(r)\,r^{-\frac{n-(s-t_0)p}{p}} \leq C\,\mathscr{H}(r)\,r^{-\frac{n-(s-t_0)p}{p}}
 \end{split}
 \end{equation}
So, assuming the notation introduced in \eqref{B},
by combining together \eqref{1add-H'}, \eqref{2add H' q} and \eqref{2add H' p}, and recalling that $\Theta<p^*$, we have 
\begin{equation*}
|h(r)|\leq C'\mathscr{H}(r)\, r^{-\frac{n}{p*}} + \bar{C}\mathscr{H}(r)\,r^{-\frac{n}{\Theta}} \leq C_2\mathscr{H}(r)\,r^{-\frac{n}{\Theta}}.
\end{equation*}
\end{proof}

\begin{oss}
In \cite[Proposition 3.6]{FGS} Focardi, Gelli and Spadaro showed the estimate \eqref{estimate H'} with $h(r)$ given in \eqref{hr}. 
Under the hypothesis $\A\in W^{1,\infty}(\Om,\R^n\times~\R^n)$, 
the term $\diw\A(x)\in L^\infty(\Om, \R^n)$, so they obtained $|h(r)|\leq C \mathscr{H}(r)$.
\end{oss}

\subsection{Proof of Weiss’s quasi-monotonicity formula}
In this section we prove a Weiss’ quasi-monotonicity formula that is one of the main results of the paper. 
The plan of proof is the same as \cite[Theorem 3.7]{FGS}. The difference, due to the lack of regularity of the coefficients,
consists in the presence of additional unbounded factors. These factors are produced in 
Proposition \ref{E'}, Proposition \ref{H'}, and from a freezing argument, and they include: 
$r^{-\frac{n}{p^*}}$, $r^{-\frac{n}{\Theta}}$ and $\frac{\omega(r)}{r}$.
The key observation is that, for our hypotheses, these terms are integrable, so we are able to obtain the formula.
For completeness we report the proof with all the details.

\begin{proof}[Proof of Theorem \ref{Weiss}]
Assume the definition of $\Phi(r)$ by \eqref{Phir}:
 \begin{equation*}
 \Phi(r):=r^{-n-2}\mathcal{E}(r) - 2\,r^{-n-3}\mathscr{H}(r).
\end{equation*}
Then for $\mathcal{L}^1$-a.e. $r\in \mathrm{dist}(\underline{0},\partial\Om)$ we have
\begin{equation}
 \Phi'(r)=\frac{\mathcal{E}'(r)}{r^{n+2}}-(n+2)\frac{\mathcal{E}(r)}{r^{n+3}} - 2\frac{\mathscr{H}'(r)}{r^{n+3}} +2(n+3)\frac{\mathscr{H}(r)}{r^{n+4}}.
\end{equation}
By Proposition \ref{E'} we have
\begin{equation*}
 \begin{split}
  &\frac{\mathcal{E}'(r)}{r^{n+2}}-(n+2)\frac{\mathcal{E}(r)}{r^{n+3}}\geq \frac{2}{r^{n+2}}\int_{\partial B_r}\mu^{-1}\langle \A\nu, \nabla u \rangle^2\,d\mathcal{H}^{n-1}
   + \frac{1}{r^{n+3}}\int_{B_r}\langle\A\nabla u,\nabla u\rangle\,\diw(\mu^{-1}\A x)\,dx\\ 
   &- \frac{2}{r^{n+3}}\int_{B_r}f \langle\mu^{-1}\A x,\nabla u\rangle\,dx
   - \frac{2}{r^{n+3}}\int_{B_r}\langle\A\nabla u, \nabla^T(\mu^{-1}\A x)\nabla u\rangle\, dx + \frac{2}{r^{n+2}}\int_{\partial B_r} f u\, d\mathcal{H}^{n-1}\\
   &-\frac{C_1}{r^{n+2}}\frac{\mathcal{E}(r)}{r^{\frac{n}{p^*}}}
   - \frac{n+2}{r^{n+3}}\int_{B_r}\langle\A\nabla u, \nabla u\rangle\, dx - \frac{2(n+2)}{r^{n+3}}\int_{B_r} f u\,dx.  
 \end{split}
\end{equation*}
Then, integrating by parts and given \eqref{PDE_u}:
\begin{equation}\label{int-part1-f.mon-Weiss}
 \begin{split}
  \int_{B_r}\langle\A\nabla u, \nabla u\rangle\, dx + \int_{B_r} f u\,dx &= \int_{B_r}\langle\A\nabla u, \nabla u\rangle\, dx + \int_{B_r} u\diw(\A\nabla u)\,dx
  = \int_{\partial B_r} u \langle \A\nu, \nabla u \rangle \, d\mathcal{H}^{n-1}.
 \end{split}
\end{equation}
Thus, applying \eqref{int-part1-f.mon-Weiss} in four occurrences, we deduce
\begin{equation}\label{E'-E}
 \begin{split}
  &\frac{\mathcal{E}'(r)}{r^{n+2}}-(n+2)\frac{\mathcal{E}(r)}{r^{n+3}}\geq 
    - \frac{C_1}{r^{n+2}}\mathcal{E}(r)\,r^{-\frac{n}{p^*}} 
    + \frac{2}{r^{n+2}}\int_{\partial B_r}\mu^{-1}\langle \A\nu, \nabla u \rangle^2\,d\mathcal{H}^{n-1}\\
   &+ \frac{1}{r^{n+3}}\int_{B_r}\!\!\!\!\Big(\langle\A\nabla u,\nabla u\rangle\,\diw(\mu^{-1}\A x) - 2\langle\A\nabla u, \nabla^T(\mu^{-1}\A x)\nabla u\rangle - (n-2)\langle\A\nabla u, \nabla u\rangle\Big)\,dx\\ 
   & - \frac{2}{r^{n+3}}\int_{B_r}\!\!f \langle\mu^{-1}\A x,\nabla u\rangle\,dx + \frac{2}{r^{n+2}}\int_{\partial B_r}\!\! f u\, d\mathcal{H}^{n-1} -\frac{4}{r^{n+3}}\int_{\partial B_r}\!\! u \langle \A\nu, \nabla u \rangle \, d\mathcal{H}^{n-1}
    - \frac{2n}{r^{n+3}}\int_{B_r}\!\! f u\,dx.  
 \end{split}
\end{equation}
Instead the Proposition \ref{H'} leads to
\begin{equation}\label{H'-H}
 \begin{split}
 -2\frac{\mathscr{H}'(r)}{r^{n+3}} + 2(n+3)\frac{\mathscr{H}(r)}{r^{n+4}}\geq 
 - \frac{2 C_2}{r^{n+3}} \mathscr{H}(r)r^{-\frac{n}{\Theta}} + \frac{8}{r^{n+4}}\mathscr{H}(r) - \frac{4}{r^{n+3}}\int_{\partial B_r} u \langle \A\nu,\nabla u \rangle\, d\mathcal{H}^{n-1}.
 \end{split}
\end{equation}
By combining together \eqref{E'-E} and \eqref{H'-H} and since $p^*\geq \Theta$ we finally infer that
\begin{equation}\label{Phi' + Phi}
 \begin{split}
 &\Phi'(r) + (C'_1 \vee C_2) \Phi(r) r^{-\frac{n}{\Theta}}\geq \frac{2}{r^{n+2}}\int_{\partial B_r}\mu^{-1}\langle \A\nu, \nabla u \rangle^2\,d\mathcal{H}^{n-1}\\
   &+ \frac{1}{r^{n+3}}\int_{B_r}\!\!\!\!\Big(\langle\A\nabla u,\nabla u\rangle\,\diw(\mu^{-1}\A x) - 2\langle\A\nabla u, \nabla^T(\mu^{-1}\A x)\nabla u\rangle - (n-2)\langle\A\nabla u, \nabla u\rangle\Big)\,dx\\ 
   & - \frac{2}{r^{n+3}}\int_{B_r}f \langle\mu^{-1}\A x,\nabla u\rangle\,dx + \frac{2}{r^{n+2}}\int_{\partial B_r} f u\, d\mathcal{H}^{n-1} -\frac{4}{r^{n+3}}\int_{\partial B_r} u \langle \A\nu, \nabla u \rangle \, d\mathcal{H}^{n-1}\\
   & - \frac{2n}{r^{n+3}}\int_{B_r} f u\,dx + \frac{8}{r^{n+4}}\mathscr{H}(r) - \frac{4}{r^{n+3}}\int_{\partial B_r} u \langle \A\nu,\nabla u \rangle\, d\mathcal{H}^{n-1}\\
   &= \frac{2}{r^{n+2}}\int_{\partial B_r}\Big(\mu^{-1}\langle \A\nu, \nabla u \rangle^2 + 4\frac{u^2}{r^2}\mu - 4\frac{u}{r}\langle \A\nu, \nabla u \rangle\Big)\,d\mathcal{H}^{n-1}\\
   &+ \frac{1}{r^{n+3}}\int_{B_r}\!\!\!\!\Big(\langle\A\nabla u,\nabla u\rangle\,\diw(\mu^{-1}\A x) - 2\langle\A\nabla u, \nabla^T(\mu^{-1}\A x)\nabla u\rangle - (n-2)\langle\A\nabla u, \nabla u\rangle\Big)\,dx\\ 
   & - \frac{2}{r^{n+3}}\bigg(\int_{B_r}f \big(\langle\mu^{-1}\A x,\nabla u\rangle + nu\big)\,dx - r\int_{\partial B_r} f u\,d\mathcal{H}^{n-1}\bigg) 
   =: R_1 + R_2 + R_3.
 \end{split}
\end{equation}
We can estimate the three addenda separately.
\begin{equation}
 \begin{split}\label{R_1}
  R_1&= \frac{2}{r^{n+2}}\int_{\partial B_r}\mu \Big(\mu^{-2}\langle \A\nu, \nabla u \rangle^2 + 4\frac{u^2}{r^2} - 4\mu^{-1}\frac{u}{r}\langle \A\nu, \nabla u \rangle\Big)\,d\mathcal{H}^{n-1}\\
  &=\frac{2}{r^{n+2}}\int_{\partial B_r}\mu \Big(\langle \mu^{-1} \A\nu, \nabla u \rangle - 2\frac{u}{r} \Big)^2\,d\mathcal{H}^{n-1}.
  \end{split}
\end{equation}
Since $n= \diw(x)$ by \eqref{u e grad u limitate} we have
\begin{equation*}
 \begin{split}
  |R_2|&= \frac{1}{r^{n+3}}\left|\int_{B_r}\!\!\!\!\Big(\langle\A\nabla u,\nabla u\rangle\,\diw(\mu^{-1}\A x) - 2\langle\A\nabla u, \nabla^T(\mu^{-1}\A x)\nabla u\rangle - (n-2)\langle\A\nabla u, \nabla u\rangle\Big)\,dx\right|\\ 
     &= \left|\frac{1}{r^{n+3}}\int_{B_r}\!\!\!\!\Big(\langle\A\nabla u,\nabla u\rangle\,\diw(\mu^{-1}\A x- x) - 2\langle\A\nabla u, \nabla^T(\mu^{-1}\A x - x)\nabla u\rangle \Big)\,dx\right|\\
     &\leq \frac{C^2\,\Lambda}{r^{n+1}}\int_{B_r}\!\!\!\!\Big(|\diw(\mu^{-1}\A x- x) + 2|\nabla(\mu^{-1}\A x - x)|\Big)\,dx
     \leq \frac{C'\,\Lambda}{r^{n+1}}\int_{B_r}\!\!\!\!\Big(|\nabla(\mu^{-1}\A x - x)|\Big)\,dx,
 \end{split}
\end{equation*}
We estimate $|\nabla(\mu^{-1}\A x - x)|$:
 \begin{equation*}
 \begin{split}
  |\nabla(\mu^{-1}\A x - x)|&= \left|\nabla\left((\mu^{-1}\A - id)x\right)\right|=|\nabla(\mu^{-1}\A - id)x + (\mu^{-1}\A - id)|\\ 
  &=|\nabla(\mu^{-1})\otimes \A x + \mu^{-1} \nabla\A\, x + (\mu^{-1}\A - I_n)| \leq \Lambda r (|\nabla \A|+|\nabla \mu|) + C\,r^{1-\frac{n}{p^*}},
 \end{split}
\end{equation*}
where in the last inequality, we have used the $\c$-Holder continuity of $\A-\mu I_n$. Thus, from Lemma \ref{lem reg mu}
\begin{equation*}
 \begin{split}
  |R_2|&\leq \frac{C'\Lambda}{r^{n+1}}\int_{B_r}\Big(|\nabla(\mu^{-1}\A x - x)|\Big)\,dx
  \leq \frac{C''}{r^{n+1}}\int_{B_r}\left(r^{1-\frac{n}{p^*}}+r(|\nabla \A|+|\nabla \mu|)\right)\,dx \\
  &\leq C'''r^{-\frac{n}{p^*}} + \frac{C''}{r^n}\bigg(\Big(\int_{B_r}|\nabla\A|^{p^*}\,dx\Big)^\frac{1}{p*}(\omega_n r^n)^{1-\frac{1}{p^*}}+\Big(\int_{B_r}|\nabla\mu|^{q}\,dx\Big)^\frac{1}{q}(\omega_n r^n)^{1-\frac{1}{q}}\bigg)
  \leq \bar{C}r^{-\frac{n}{q}},
 \end{split}
\end{equation*}
for each $n<\Theta<q<p^*$, whence 
\begin{equation*}
 |R_2|\leq c\, \frac{\mathcal{E}(r)}{r^{n+2}}r^{-\frac{n}{q}}.
\end{equation*}
Moreover, from \eqref{Hr} and \eqref{mu limitata}
\begin{equation*}
 0\leq \frac{\mathscr{H}(r)}{r^{n+3}}\leq c\|u_r\|_{L^\infty}^2\leq c,
\end{equation*}
with a certain constant $c$ independent from $r$,
then
\begin{equation}\label{R_2}
 |R_2|\leq c\,\left(\frac{\mathcal{E}(r)}{r^{n+2}}-2\frac{\mathscr{H}(r)}{r^{n+3}}\right)r^{-\frac{n}{q}} + 2c\,\frac{\mathscr{H}(r)}{r^{n+3}}\, r^{-\frac{n}{q}}\leq c\,\Phi(r)r^{-\frac{n}{\Theta}}+c\,r^{-\frac{n}{\Theta}}.
\end{equation}
Finally, assuming that $n=\diw x$ and using the following identity, which is consequence of the divergence theorem
\begin{equation}
 \int_{B_r}\left(\langle x,\nabla u\rangle + u\,\diw x\right)\,dx = r\int_{\partial B_r} u\,d\mathcal{H}^{n-1}, 
\end{equation}
we have 
\begin{equation*}
 \begin{split}
R_3=& - \frac{2}{r^{n+3}}\bigg(\int_{B_r}f \big(\langle\mu^{-1}\A x,\nabla u\rangle + nu\big)\,dx - r\int_{\partial B_r} f u\,d\mathcal{H}^{n-1}\bigg) \\ 
=& - \frac{2}{r^{n+3}}\bigg(\int_{B_r}(f(x)-f(\underline{0})) \big(\langle\mu^{-1}\A x,\nabla u\rangle +nu\big)\,dx + f(\underline{0})\int_{B_r} \big(\langle\mu^{-1}\A x,\nabla u\rangle - u\,\diw x\big)\,dx\\ 
&- r\int_{\partial B_r} (f(x)-f(\underline{0})) u\,d\mathcal{H}^{n-1} - r\,f(\underline{0})\int_{\partial B_r} u\,d\mathcal{H}^{n-1} \bigg) \\
=& - \frac{2}{r^{n+3}}\bigg(\int_{B_r}(f(x)-f(\underline{0})) \big(\langle\mu^{-1}\A x,\nabla u\rangle +nu\big)\,dx + f(\underline{0})\int_{B_r} \big(\langle\mu^{-1}\A x - x,\nabla u\rangle \big)\,dx\\ 
&- r\int_{\partial B_r} (f(x)-f(\underline{0})) u\,d\mathcal{H}^{n-1} \bigg).
 \end{split}
\end{equation*}
Thus
\begin{equation}\label{R_3}
 \begin{split}
|R_3|=& \frac{2}{r^{n+3}}\bigg| f(\underline{0})\int_{B_r}\big(\langle\mu^{-1}\A x - x,\nabla u\rangle \big)\,dx\\
    &+\int_{B_r}(f(x)-f(\underline{0})) \big(\langle\mu^{-1}\A x,\nabla u\rangle +nu\big)\,dx 
     - r\int_{\partial B_r} (f(x)-f(\underline{0})) u\,d\mathcal{H}^{n-1} \bigg|\\
 \leq& \frac{c}{r^{n+1}} \bigg(\int_{B_r}|\A-\mu I_n|\,dx + \int_{B_r}|f(x)-f(\underline{0})|\,dx + r\int_{\partial B_r} |f(x)-f(\underline{0})|\,d\mathcal{H}^{n-1}\bigg)\\
 \leq& \frac{c}{r^{n+1}}(r^{n+1-\frac{n}{p^*}}+r^n\,\omega(r))\leq c \left(r^{-\frac{n}{p^*}}+ \frac{\omega(r)}{r}\right).
 \end{split}
\end{equation}
Now by combining together \eqref{Phi' + Phi}, \eqref{R_1}, \eqref{R_2} and \eqref{R_3}
we have
\begin{equation}
 \begin{split}
\Phi'(r) + C_3 \Phi(r)\,r^{-\frac{n}{\Theta}} + C_4\, \left(r^{-\frac{n}{\Theta}}+\frac{\omega(r)}{r}\right)\geq \frac{2}{r^{n+2}}\int_{\partial B_r}\mu \Big(\langle \mu^{-1} \A\nu, \nabla u \rangle - 2\frac{u}{r} \Big)^2\,d\mathcal{H}^{n-1}.
 \end{split}
\end{equation}
Multiplying the inequality by the integral factor $e^{\bar{C_3}r^{1-\frac{n}{\Theta}}}$ with $\bar{C_3}= \frac{C_3}{1-\frac{n}{\Theta}}$
we get
\begin{equation*}
 \left(\Phi(r)\,e^{\bar{C_3}r^{1-\frac{n}{\Theta}}}\right)' + C_4\,\left(r^{-\frac{n}{\Theta}}+\frac{\omega(r)}{r}\right)e^{\bar{C_3}r^{1-\frac{n}{\Theta}}}
 \geq \frac{2e^{\bar{C_3}r^{1-\frac{n}{\Theta}}}}{r^{n+2}}\int_{\partial B_r}\mu \Big(\langle \mu^{-1} \A\nu, \nabla u \rangle - 2\frac{u}{r} \Big)^2\,d\mathcal{H}^{n-1}
\end{equation*}
whence
\begin{equation}
 \begin{split}
 \frac{d}{dr}\bigg(\Phi(r)\,e^{\bar{C_3}r^{1-\frac{n}{\Theta}}} &+ C_4\!\!\!\int_0^r \!\left(t^{-\frac{n}{\Theta}}+\frac{\omega(t)}{t}\right)e^{\bar{C_3}t^{1-\frac{n}{\Theta}}}\,dt\bigg)
 \geq \frac{2e^{\bar{C_3}r^{1-\frac{n}{\Theta}}}}{r^{n+2}}\int_{\partial B_r}\!\!\!\mu \Big(\langle \mu^{-1} \A\nu, \nabla u \rangle - 2\frac{u}{r} \Big)^2\,d\mathcal{H}^{n-1}.
 \end{split}
\end{equation}
In particular, the quantity under the sign of the derivative, bounded by construction, is also monotonic, therefore  its limit exists
as $r\to 0^+$. It follows that $\Phi(0^+):=\lim_{r\to 0^+}\Phi(r)$ exists and is bounded.

Finally
\begin{equation}
 \begin{split}
 \Phi(r)-\Phi(0^+)\geq& -|\Phi(r)\,e^{\bar{C_3}r^{1-\frac{n}{\Theta}}}-\Phi(r)| + \Phi(r)\,e^{\bar{C_3}r^{1-\frac{n}{\Theta}}} + C_4\,\int_0^r \left(t^{-\frac{n}{\Theta}}+\frac{\omega(t)}{t}\right)e^{\bar{C_3}t^{1-\frac{n}{\Theta}}}\,dt\\
  &-\Phi(0^+) - C_4\,\int_0^r \left(t^{-\frac{n}{\Theta}}+\frac{\omega(t)}{t}\right)e^{\bar{C_3}t^{1-\frac{n}{\Theta}}}\,dt\\
  \geq& -|\Phi(r)|\,c'\,r^{1-\frac{n}{\Theta}} + \Phi(r)\,e^{\bar{C_3}r^{1-\frac{n}{\Theta}}} + C_4\,\int_0^r \left(t^{-\frac{n}{\Theta}}+\frac{\omega(t)}{t}\right)e^{\bar{C_3}t^{1-\frac{n}{\Theta}}}\,dt\\
  &-\Phi(0^+) - c'\,\left(r^{1-\frac{n}{\Theta}}+\int_0^r\frac{\omega(t)}{t}\,dt\right)\\
  \geq &\Phi(r)\,e^{\bar{C_3}r^{1-\frac{n}{\Theta}}} + C_4\,\int_0^r \left(t^{-\frac{n}{\Theta}}+\frac{\omega(t)}{t}\right)e^{\bar{C_3}t^{1-\frac{n}{\Theta}}}\,dt -\Phi(0^+) - c\,\left(r^{1-\frac{n}{\Theta}}+\int_0^r\frac{\omega(t)}{t}\,dt\right),
 \end{split}
\end{equation}
where in the last inequality, we used the boundedness of $\Phi(r)$.
\end{proof}

\begin{oss}
In \cite[Theorem 3.7]{FGS}, under the hypotheses $\A\in W^{1,\infty}(\Om,\R^n\times~\R^n)$ and $f\in C^{0,\a}(\Om)$, 
Focardi, Gelli and Spadaro proved that the following estimate holds true for 
$\mathcal{L}^1$-a.e. $r$ in $(0,\frac{1}{2}\mathrm{dist}(\underline{0},\partial\Om)\wedge~1)$:
\begin{equation*}
 \frac{d}{dr}\bigg(\Phi(r)\,e^{\bar{C_3}r} + C_4\!\!\!\int_0^r \!t^{\a-1}e^{\bar{C_3}t}\,dt\bigg)
 \geq \frac{2e^{\bar{C_3}r}}{r^{n+2}}\int_{\partial B_r}\!\!\!\mu \Big(\langle \mu^{-1} \A\nu, \nabla u \rangle - 2\frac{u}{r} \Big)^2\,d\mathcal{H}^{n-1}.
 \end{equation*}
\end{oss}

\begin{oss}\label{oss cost bdd in cpt}
 We note that from Proposition \ref{prop u_r limitata W2p} the uniform boundedness of the sequence $(u_{x_0,r})_r$
 in $C^{0,\c}(\R^n)$ follows. Moreover, for base points $x_0$ in a compact set of $\Om$, the $C^{0,\c}$ norms, and thus 
 the constants in the monotonicity formulae, are uniformly bounded.  
 Indeed, as pointed out in the corresponding statements they depend on $\|\A\|_{W^{s,p}(\Om)}$ and 
 $\mathrm{dist}(x_0,\partial\Om)$.
\end{oss}

\section{The blow-up method: Classification of blow-ups}\label{s:classification}
In this section we proceed with the analysis of blow-ups showing the consequences of Theorem \ref{Weiss}.

The first consequence is that the blow-ups are $2$-homogeneous, i.e. $v(tx)=t^2v(x)$ for all $t>0$ and for all $x\in \R^n$, 
as it is possible to deduce from the second member of \ref{dis monotonia Weiss} where, 
according to Euler's homogeneous function Theorem\footnote{Let $v:\R^n\to \R$ a differentiable function, 
then $v$ is $k$-homogeneous with $k>0$ if and only if $k\,v(x)=\langle \nabla v(x), x \rangle$.}, 
the integral represents a distance to a $2$-homogeneous function set.
For a proof of the following result we refer to \cite[Proposition 4.2]{FGS}.
\begin{prop}[$2$-homogeneity of blow-ups]\label{prop blow-up 2omo}
 Let $x_0\in \Gamma_u$ and $(u_{x_0,r})_r$ as in (\ref{u_x_0 r}). Then, for every sequence $(r_j)_j\downarrow 0$ there exists
 a subsequence $(r_{j_k})_k\subset(r_j)_j$ such that the sequence $(u_{x_0,r_{j_k}})_k$ converges in $C^{1,\c}(\R^n)$
 to a function $v(y)=w(\LL^{-1}(x_0)y)$, where $w$ is $2$-homogeneous.
\end{prop}

As a second consequence, remembering Proposition \ref{prop crescita basso} we can obtain that the blow-ups are nonzero.
 \begin{cor}\label{cor w non 0}
 Let $v(y)=w(\LL^{-1}(x_0)y)$ be a limit of $C^{1,\c}$ a converging sequence of rescalings $(u_{x_0, r_j})_j$ in a free-boundary point 
  $x_0\in \Gamma_u$, then $\underline{0}\in \Gamma_w$, i.e. $w\not\equiv 0$ in any neighborhood of $\underline{0}$.
\end{cor}
\begin{proof}
 Due to Proposition \ref{prop crescita basso} for any $j\in \N$, there exists a $\nu_j\in \Sf^{n-1}$ such that $u_{x_0,r_j}(\nu_j)\geq \theta$.
 From the compactness of $\Sf^{n-1}$ we can extract a subsequence $(\nu_{j_k})_k$ such that $\nu_{j_k}\to \nu\in \Sf^{n-1}$. 
 Due to the convergence in $C^{1,\c}$ we have that $v(\nu)\geq \theta$, if we define $\xi:=\LL^{-1}(x_0)\nu$, we get $w(\xi)\geq \theta$. 
 As noticed in Proposition \ref{prop blow-up 2omo} $w$ is $2$-homogeneous, then in any neighborhood of $\underline{0}$ there exists a point on the 
 direction $\xi$ on which $w$ is strictly positive, so for any $\d>0$ we have $w(\d\xi)=\d^2 w(\xi)\geq \d^2 \theta$,
 and thus this Corollary is verified.
\end{proof}

Finally, it is possible to give a classification of blow-ups. We begin by recalling the result in the classical case established 
by Caffarelli \cite{Caf77,Caf80,Caf98}. 
\begin{defi}
 A \emph{global solution} to the obstacle problem is a positive function  $w\in C^{1,1}_{loc}(\R^n)$ solving \eqref{PDE_u} 
 in the case $\A\equiv I_n$ and $f\equiv 1$.
\end{defi}
The following result occurs:
\begin{teo}\label{teo Caffarelli}
 Every global solution $w$ is convex. Moreover, if $w\not\equiv 0$ and $2$-homogeneous, then one of the following two cases occurs:
 \begin{itemize}
  \item [(A)] $w(y)=\frac{1}{2}\big(\langle y, \nu\rangle \vee 0\big)^2$ for some $\nu\in \Sf^{n-1}$,
  where the symbol $\vee$ denotes the maximum of the surrounding quantities;
  \item [(B)] $w(y)=\langle \mathbb{B}y, y\rangle$ with  $\mathbb{B}$ a symmetric, positive semidefinite matrix satisfying  
  and $\mathrm{Tr}\mathbb{B}=\frac{1}{2}$.
 \end{itemize}
\end{teo}

Having this result at hand, a complete classification of the blow-up limits, for the obstacle
problem \eqref{problema ad ostacolo}, follows as in the classical context. 
Up to minimal difference the result looks like \cite[Proposition 4.5]{FGS} to which we refer for the proof.
The main ingredients of the proof are the quasi-monotonicity formula by Weiss 
and a $\Gamma$-convergence argument:

\begin{prop}[Classification of blow-ups]\label{prop Gamma conv}
 Every blow-up $v_{x_0}$ at a free-boundary point $x_0\in \Gamma_u$ is of the form $v_{x_0}=w(\LL^{-1}(x_0)y)$, with $w$ a non-trivial, 
 $2$-homogeneous global solution.
\end{prop}
According to Theorem \ref{teo Caffarelli} we shall call a global solution of type $(A)$ or of type $(B)$.

The above proposition allows us to formulate a simple criterion to distinguish between regular and singular free-boundary points.
\begin{defi}\label{reg sing}
 A point $x_0\in \Gamma_u$ is a  \emph{regular} free-boundary point, and we write $x_0\in Reg(u)$ if there exists a blow-up 
 of $u$ at $x_0$ of type $(A)$. Otherwise, we say that $x_0$ is \emph{singular} and write $x_0\in Sing(u)$.
\end{defi}

\begin{oss}\label{oss conto energia}
 Simple calculations show that $\Psi_w(1)=\theta$ for every global solution of type $(A)$ and $\Psi_w(1)=2\theta$ for every global solution 
 of type $(B)$, where $\Psi_w$ is the energy defined in \eqref{Psi_v} and $\theta$ is a dimensional constant.
\end{oss}

\begin{oss}
We observe that for every sequence $r_j\searrow0$ for which $u_{\LL(x_0),r_j}\to w$ in $C^{1,\c}(B_1)$ with $w$ being a 
 $2$-homogeneous global solution then
 \begin{equation*}
  \lim_{r_j\to 0} \Phi_{\LL(x_0)}(r_j)=\Psi_w(1).
 \end{equation*}
 From Weiss’ quasi-monotonicity the uniqueness of the limit follows, so $\Phi_{\LL(x_0)}(0)=\Psi_w(1)$ for every $w$ that is 
 the limit of the sequence $(u_{\LL(x_0),r})_{r}$. 
 It follows that if $x_0\in \Gamma_u$ is a regular point then $\Phi_{\LL(x_0)}(0)=\theta$ or, 
 equivalently every blow-up at $x_0$ is of type $(A)$.
\end{oss}

\section{Monneau’s quasi-monotonicity formula}
In this section we prove a Monneau type quasi-monotonicity formula (see \cite{Monneau}) for singular free-boundary points. 
The plan of proof follows \cite[Theorem 3.8]{FGS}. The additional difficulty is the same as 
Theorem \ref{Weiss} so for completeness we report the whole proof.

Let $v$ be a $2$-homogeneous positive polynomial, solving
\begin{equation}\label{laplaciano v=1}
 \Delta v = 1 \qquad\qquad \mathrm{on}\,\, \R^n.
\end{equation}
Let
\begin{equation}\label{Psi_v}
 \Psi_v(r):=\frac{1}{r^{n+2}}\int_{B_r}\big(|\nabla v|^2 + 2v\big)\,dx -\frac{2}{r^{n+3}}\int_{\partial B_r} v^2\,d\mathcal{H}^{n-1}.
\end{equation}
We note that the expression of $\Psi_v(r)$ is analogous to those of $\Phi$ with coefficients frozen in $\underline{0}$ (recalling \eqref{Phir}).
An integration by parts, \eqref{Psi_v} and the $2$-homogeneity of $v$ yields
\begin{equation*}
\begin{split}
 \frac{1}{r^{n+2}}\int_{B_r}|\nabla v|^2\,dx=&\frac{1}{r^{n+2}}\int_{B_r}\big(\diw(v\nabla v) - v\,\Delta v\big)\,dx
 =\frac{1}{r^{n+3}}\int_{\partial B_r} \langle\nabla v,x\rangle\,d\mathcal{H}^{n-1} - \frac{1}{r^{n+2}}\int_{B_r}v\,dx\\
 =&\frac{1}{r^{n+3}}\int_{\partial B_r} v^2\,d\mathcal{H}^{n-1} - \frac{1}{r^{n+2}}\int_{B_r}v\,dx
 =\int_{\partial B_1} v^2\,d\mathcal{H}^{n-1} - \int_{B_1}v\,dx
\end{split}
\end{equation*}
and therefore 
\begin{equation}\label{Psi_v = int v}
 \Psi_v(r)=\Psi_v(1)=\int_{B_1}v\,dx.
\end{equation}
In the next theorem we give a monotonicity formula for solutions of the obstacle problem such that $\underline{0}$
is a point of the free-boundary and 
\begin{equation}\label{pto sing1}
 \Phi(0^+)=\Psi_v(1) \quad \textrm{for some}\, v\, 2\textrm{-homogeneous solution of \eqref{laplaciano v=1}}. 
\end{equation}
As explained in Definition \ref{reg sing}, formula \eqref{pto sing1} characterizes the singular part of the free
boundary.

\begin{proof}[Proof of Theorem \ref{Monneau}]
Set $w_r=u_r-v$. As $v$ is $2$-homogenus we have that $w_r(x)=\frac{w(rx)}{r^2}$. Assuming that from \eqref{abuso_notazione} 
 $\A(\underline{0})=I_n$, due to the Divergence Theorem and Euler's homogeneous function Theorem we find
 \begin{equation*}
  \begin{split}
  \frac{d}{dr}\int_{\partial B_1} w_r^2\,&d\mathcal{H}^{n-1} = \int_{\partial B_1} w_r\,\frac{d}{dr}\left(\frac{w(rx)}{r^2}\right)\,d\mathcal{H}^{n-1}\\ 
  =& \frac{2}{r}\int_{\partial B_1} w_r (\langle\nabla w_r, x\rangle - 2w_r)\,d\mathcal{H}^{n-1} 
  = \frac{2}{r}\int_{\partial B_1} w_r (\langle\nabla u_r, x\rangle - 2u_r)\,d\mathcal{H}^{n-1}\\ 
  =& \frac{2}{r}\int_{\partial B_1} w_r (\langle \A(rx)\nabla u_r, x\rangle - 2u_r)\,d\mathcal{H}^{n-1} 
  +\frac{2}{r}\int_{\partial B_1} w_r \langle (\A(\underline{0})-\A(rx))\nabla u_r, x\rangle \,d\mathcal{H}^{n-1}\\
  \geq& \frac{2}{r}\int_{\partial B_1} w_r (\langle \A(rx)\nabla u_r, x\rangle - 2u_r)\,d\mathcal{H}^{n-1} 
  - C\, \|\nabla u_r\|_{L^2(\partial B_1)}\, \|w_r\|_{L^2(\partial B_1)}\, [\A]_{0,\c}\,r^{-\frac{n}{p^*}},
  \end{split}
 \end{equation*}
 thus by \eqref{u e grad u limitate} 
 \begin{equation}\label{monneau1}
  \begin{split}
  \frac{d}{dr}\int_{\partial B_1} w_r^2\,d\mathcal{H}^{n-1} \geq \frac{2}{r}\,\int_{\partial B_1} w_r (\langle \A(rx)\nabla u_r, x\rangle - 2u_r)\,d\mathcal{H}^{n-1} - C\,r^{-\frac{n}{p^*}}.
  \end{split}
 \end{equation}
 Using an integration by parts,  and \eqref{laplaciano v=1} we can rewrite the first term on the right as
 \begin{equation}
  \begin{split}
  \int_{\partial B_1}& w_r (\langle \A(rx)\nabla u_r, x\rangle - 2u_r)\,d\mathcal{H}^{n-1} \\
  \stackrel{\eqref{PDE_u}}{=}& \int_{B_1}\big(\langle \A(rx)\nabla u_r, \nabla w_r\rangle + w_r\,f(rx)\chi_{\{u_r>0\}}(x)\big)\,dx -\int_{\partial B_1} 2\,w_r\,u_r\,d\mathcal{H}^{n-1} \\
  =& \int_{B_1}\big(\langle \A(rx)\nabla u_r, \nabla u_r\rangle + u_r\,f(rx)\chi_{\{u_r>0\}}(x)\big)\,dx -\int_{\partial B_1} 2\,u_r^2\,d\mathcal{H}^{n-1}\\
  &- \int_{B_1}\big(\langle \A(rx)\nabla u_r, \nabla v\rangle + v\,f(rx)\chi_{\{u_r>0\}}(x)\big)\,dx +\int_{\partial B_1} 2\,v\,u_r\,d\mathcal{H}^{n-1} \\
  =& \Phi(r) - \int_{B_1} f(rx)\big(u_r + v\,\chi_{\{u_r>0\}}(x)\big)\,dx + 2\int_{\partial B_1} \big(\mu(rx)-\mu(\underline{0})\big)u_r^2\,d\mathcal{H}^{n-1}\\
  &- \int_{B_1}\langle \A(rx)\nabla u_r, \nabla v\rangle\,dx +2\int_{\partial B_1} v\,u_r\,d\mathcal{H}^{n-1}\\  
  \geq& \Phi(r) - \int_{B_1}(u_r + v\,\chi_{\{u_r>0\}}(x))\,dx -\int_{B_1}\langle \nabla u_r, \nabla v\rangle\,dx
  - \int_{B_1} \big(f(rx)-f(\underline{0})\big)(u_r + v\,\chi_{\{u_r>0\}}(x))\,dx\\ 
  &- \int_{B_1}\langle \big(\A(rx)-\A(\underline{0})\big)\nabla u_r, \nabla v\rangle\,dx 
  + 2\int_{\partial B_1} \big(\mu(rx)-\mu(\underline{0})\big)u_r^2\,d\mathcal{H}^{n-1} 
   +2\int_{\partial B_1} v\,u_r\,d\mathcal{H}^{n-1}.
\end{split}
 \end{equation}    
Recalling the $\c$-H\"{o}lder continuity of $\A$ and $\mu$, from the Divergence Theorem, we obtain
\begin{equation}\label{monneau2}
  \begin{split}
  \int_{\partial B_1}& w_r (\langle \A(rx)\nabla u_r, x\rangle - 2u_r)\,d\mathcal{H}^{n-1} \\
  \geq& \Phi(r) - \int_{B_1}(u_r + v)\,dx -\int_{B_1}\langle \nabla u_r, \nabla v\rangle\,dx +2\int_{\partial B_1} v\,u_r\,d\mathcal{H}^{n-1} -c\,\left(r^\c + \omega(r)\right)\\
  \stackrel{\eqref{Psi_v = int v}}{=}& \Phi(r) -\Psi_v(1) - \int_{B_1}(u_r\,\Delta v)\,dx -\int_{B_1}\langle \nabla u_r, \nabla v\rangle\,dx +2\int_{\partial B_1} v\,u_r\,d\mathcal{H}^{n-1} -c'\,\left(r^\c + \omega(r)\right)\\
  =& \Phi(r) -\Psi_v(1) - \int_{B_1}\diw(u_r\,\nabla v)\,dx +2\int_{\partial B_1} v\,u_r\,d\mathcal{H}^{n-1} -c'\,\left(r^\c + \omega(r)\right)\\
  =& \Phi(r) -\Psi_v(1) + \int_{\partial B_1}\!\!\!u_r\big(2v - \langle\nabla v, x\rangle\big)\,d\mathcal{H}^{n-1} -c'\left(r^\c + \omega(r)\right)
  = \Phi(r) -\Psi_v(1) -c'\left(r^\c + \omega(r)\right).\\
  \end{split}
 \end{equation} 
 So, combining together \eqref{monneau1} and \eqref{monneau2}, and assuming that $\c:= 1-\frac{n}{p^*}$ we deduce
 \begin{equation*}
  \begin{split}
  \frac{d}{dr}\int_{\partial B_1} w_r^2\,d\mathcal{H}^{n-1} \geq \frac{2}{r} \big(\Phi(r) -\Psi_v(1)\big) -c'\,\left(r^{-\frac{n}{p^*}}+\frac{\omega(r)}{r}\right).
  \end{split}
 \end{equation*} 
 from inequality \eqref{weiss stima Phir} we deduce 
 \begin{equation*}
  \begin{split}
  \frac{d}{dr}\int_{\partial B_1}\!\!\!\!\!\! w_r^2\,d\mathcal{H}^{n-1} \geq& \frac{2}{r} \Big(\Phi(r)e^{\bar{C_3}r^{1-\frac{n}{\Theta}}} + C_4\,\int_0^r\!\!\! \left(t^{-\frac{n}{\Theta}}+\frac{\omega(t)}{t}\right)e^{\bar{C_3}t^{1-\frac{n}{\Theta}}}\,dt\\
      &\phantom{AAAAAAAAA}- c\,\left(r^{1-\frac{n}{\Theta}} + \int_0^r\frac{\omega(t)}{t}\,dt\right)-\Psi_v(1)\Big) -c'\,\left(r^{-\frac{n}{p^*}}+\frac{\omega(r)}{r}\right)\\
      \geq& \frac{2}{r} \Big(\Phi(r)e^{\bar{C_3}r^{1-\frac{n}{\Theta}}} + C_4\,\int_0^r\!\!\! \left(t^{-\frac{n}{\Theta}}+\frac{\omega(t)}{t}\right)e^{\bar{C_3}t^{1-\frac{n}{\Theta}}}\,dt -\Psi_v(1)\Big)\\
      &\phantom{AAAAAAAAA}-c''\,\left(r^{-\frac{n}{\Theta}}+\frac{\omega(r)}{r}+\frac{1}{r}\int_0^r\frac{\omega(t)}{t}\,dt\right)\\
  \end{split}
 \end{equation*} 
and then set $C_5=\frac{c''}{1-\frac{n}{\Theta}}$
 
 \begin{equation*}
  \begin{split}
 &\frac{d}{dr}\bigg(\int_{\partial B_1}\!\!\!\!\!\! w_r^2\,d\mathcal{H}^{n-1} + C_5\,\left(r^{1-\frac{n}{\Theta}}+\int_0^r\frac{\omega(t)}{t}\,dt + \int_0^r\frac{dt}t\int_0^t\frac{\omega(s)}{s}\,ds\right)\bigg)\\
 &\phantom{AAAAAAAAA}\geq \frac{2}{r} \Big(\Phi(r)e^{\bar{C_3}r^{1-\frac{n}{\Theta}}} + C_4\,\int_0^r\!\!\! \left(t^{-\frac{n}{\Theta}}+\frac{\omega(t)}{t}\right)e^{\bar{C_3}t^{1-\frac{n}{\Theta}}}\,dt -\Psi_v(1)\Big).
  \end{split}
 \end{equation*}  
 \end{proof}

\begin{oss}
In \cite[Theorem 3.8]{FGS}, under hypotheses $\A\in W^{1,\infty}(\Om,\R^n\times~\R^n)$ 
and $f\in C^{0,\a}(\Om)$, 
Focardi, Gelli and Spadaro proved that the following estimate holds true for
$\mathcal{L}^1$-a.e. $r$ in $(0,\frac{1}{2}\mathrm{dist}(\underline{0},\partial\Om)\wedge~1)$:
\begin{equation*}
 \frac{d}{dr}\bigg(\int_{\partial B_1}(u_r-v)^2\,d\mathcal{H}^{n-1} + C_5\,r^\a\bigg)
 \geq \frac{2}{r}\bigg(e^{C_3\,r}\Phi(r) + C_4 \int_0^r e^{C_3t}t^{\a-1}\,dt-\Psi_v(1)\bigg).
 \end{equation*}
\end{oss}
 
\section{The blow-up method: Uniqueness of blow-ups}
The last remarks show that the blow-up limits at the free-boundary points must be of a unique type: nevertheless, this does
not imply the uniqueness of the limit itself. In this paragraph we prove the property of uniqueness of blow-ups.

In view of Proposition \ref{prop Gamma conv}, if $x\in \Gamma_u$ the blow-up in $x$ is unique with form
\begin{equation*}
		v_x(y) = \left\{  \begin{array}{ll}
	\vspace{0.2cm}		  \frac{1}{2}\big(\langle \LL^{-1}(x)\varsigma(x), y\rangle \vee 0\big)^2 &		 	\quad x\in Reg(u)\\ 
				  \langle\LL^{-1}(x)\mathbb{B}_x\LL^{-1}(x) y, y\rangle &         		\quad x\in Sing(u).
			\end{array} \right.
\end{equation*}
where $\varsigma(x)\in\Sf^{n-1}$ is the blow-up direction at $x\in Reg(u)$ and $\mathbb{B}_x$ is symmetric matrix such that $\mathrm{Tr}\,\mathbb{B}_x=\frac{1}{2}$.\\

We start with the case of singular points. Therefore, from Weiss’ and Monneau’s quasi-monotonicity formulae it follows that:
\begin{prop}[{\cite[Proposition 4.11]{FGS}}]\label{unicita blowup + mod cont}
 For every point $x\in Sing(u)$ there exists a unique blow-up limit $v_x(y)=w(\LL^{-1}(x)y)$. Moreover, if $K\subset Sing(u)$ is a compact subset, 
 then, for every point $x\in K$ 
 \begin{equation}\label{modulo di continuita K pti sing}
  \left\|u_{\LL(x),r} - w \right\|_{C^1(B_1)}\leq \s_K(r) \qquad \forall r\in (0, r_K),
 \end{equation}
for some modulus of continuity $\s_K :\R^+\to\R^+$ and a radius $r_K>0$.
\end{prop}

Next, we proceed with the case of the regular points.\\
We extend the energy defined in \eqref{Psi_v} from $2$-homogeneous functions to each function $\xi\in W^{1,2}(B_1)$ by
 \begin{equation*}
  \Psi_\xi(1)=\int_{B_1}\big(|\nabla \xi|^2 + 2\xi\big)\,dx -\int_{\partial B_1} \xi^2\,d\mathcal{H}^{n-1}.
 \end{equation*}
We state Weiss' celebrated epiperimetric inequality \cite[Theorem 1]{Weiss} 
(recently a variational proof for the thin obstacle problem has been
given in \cite{FS}, and with the same approach, for the lower dimensional obstacle problem has been given in \cite{Geraci}):

\begin{teo}[Weiss’ epiperimetric inequality]\label{epip Weiss}
There exist $\d>0$ and $k\in (0,1)$ such that, for every $\phi\in H^1(B_1)$, $2$-homogeneous function, with $\|\phi-w\|_{H^1(B_1)}\leq \d$ for 
some global solution $w$ of type $(A)$, there exists a function $\xi\in H^1(B_1)$ such that 
$\xi_{|\partial B_1}=\phi_{|\partial B_1}$, $\xi\geq 0$ and 
\begin{equation}\label{dis epip}
 \Psi_\xi(1)-\theta\leq (1-k)\left(\Psi_\phi(1)-\theta\right),
\end{equation}
where $\theta=\Psi_w(1)$ is the energy of any global solution of type $(A)$.
\end{teo}

As in \cite{FGS} we prove a technical lemma that will be the key ingredient in the proof of uniqueness. With respect to \cite[Lemma 4.8]{FGS}
the lack of regularity of $\A$ and $f$ in $(H1)$-$(H3)$ does not allow us to use the final dyadic argument; 
for this reason we introduce a technical hypothesis $(H4)$ with $a>2$. For a clearer comprehension on behalf of the reader, 
we report the whole proof: 
\begin{lem}\label{lem tecn unic blow-up}
 Let $u$ be the solution of \eqref{problema ad ostacolo}  
 and we assume $(H4)$ with $a>2$ and \eqref{abuso_notazione}.
 If there exist radii $0\leq \varrho_0<r_0<1$ such that 
 \begin{equation}\label{cond inf lemma tecn}
  \inf_{w}\|{u_r}_{|\partial B_1}-w\|_{H^1(\partial B_1)}\leq \d \qquad\qquad \forall\,\,\varrho_0\leq r\leq r_0,
 \end{equation}
where the infimum is taken on all global solutions $w$ of type $(A)$ and $\d>0$ is the constant of Theorem \ref{epip Weiss},
then for each pair of radii $\varrho,t$ such that $\varrho_0\leq \varrho<t\leq r_0$ we have 
\begin{equation}\label{cond int lemma tecn}
 \int_{\partial B_1} |u_t - u_\varrho|\, d\mathcal{H}^{n-1}\leq C_7\, \rho(t),
\end{equation}
with $C_7$ a positive constant independent of  $r$ and $\varrho$, while $\rho(t)$ is a growing function vanishing in $0$.
\end{lem}
\begin{proof}
 From the Divergence Theorem, \eqref{E H-O_grande} and \eqref{H'-H} we can compute the derivative of $\Phi'(r)$ in the following way:
 \begin{equation*}
  \begin{split}
   \Phi'(r)&=\frac{\mathcal{E}'(r)}{r^{n+2}}-(n+2)\frac{\mathcal{E}(r)}{r^{n+3}} - 2\frac{\mathscr{H}'(r)}{r^{n+3}} +2(n+3)\frac{\mathscr{H}(r)}{r^{n+4}}\\
   \geq &\frac{1}{r^{n+2}}\int_{\partial B_r}(\langle\A\nabla u,\nabla u\rangle +2\,fu)\,d\mathcal{H}^{n-1} -(n+2)\frac{\mathcal{E}(r)}{r^{n+3}} + \frac{8}{r^{n+4}}\mathscr{H}(r)\\
   &- \frac{4}{r^{n+3}}\int_{\partial B_r} u \langle \A\nu,\nabla u \rangle\, d\mathcal{H}^{n-1} -C\,r^{-\frac{n}{\Theta}}\\
   \geq& \frac{1}{r^{n+2}}\int_{\partial B_r}(|\nabla u|^2 +2\,u)\,d\mathcal{H}^{n-1} -\frac{(n+2)}{r}\Phi(r) - \frac{2(n-2)}{r^{n+4}}\int_{\partial B_r} u^2\,d\mathcal{H}^{n-1}\\
   &- \frac{4}{r^{n+3}}\int_{\partial B_r} u \langle \nu,\nabla u \rangle\, d\mathcal{H}^{n-1} -C\,\left(r^{-\frac{n}{\Theta}}+\frac{\omega(r)}{r}\right)\\
   =& -\frac{(n+2)}{r}\Phi(r) +  \frac{1}{r}\int_{\partial B_1} \Big(\big(\langle\nu,\nabla u_r\rangle-2u_r\big)^2 + |\partial_\tau u_r|^2 +2u_r - 2n\,u_r^2\Big)\,d\mathcal{H}^{n-1}\\
   &- C\left(r^{-\frac{n}{\Theta}}+\frac{\omega(r)}{r}\right),
  \end{split}
 \end{equation*}
where we denote by $\partial_\tau u_r$, the tangential derivative of $u_r$ along $\partial B_1$.
Let $w_r$ be the $2$-homogeneous extension of ${u_r}_{|\partial B_1}$. 
We note that if $\varphi$ is a $2$-homogeneous function, then we have
\begin{equation}\label{cambio variabili w_r}
 \begin{split}
\int_{B_1} \varphi(x)\,dx&=\int_0^1\int_{\partial B_t} \varphi(y)\,d\mathcal{H}^{n-1}(y)\,dt=\int_0^1 t^{n+1}\int_{\partial B_1} \varphi(y)\,d\mathcal{H}^{n-1}(y)\\
&=\frac{1}{n+2}\int_{\partial B_1} \varphi(y)\,d\mathcal{H}^{n-1}(y).
 \end{split}
\end{equation}
Then a simple integration in polar coordinates, thanks to Euler's homogeneous function Theorem and \eqref{cambio variabili w_r} 
which give 
\begin{equation*}
\begin{split}
 \int_{\partial B_1}\big(|\partial_\tau & u_r|^2  +2u_r - 2n\,u_r^2\big)\,d\mathcal{H}^{n-1}=\int_{\partial B_1}\!\!\!\!\big(|\partial_\tau w_r|^2 +2w_r +4w_r^2- 2(n+2)\,w_r^2\big)\,d\mathcal{H}^{n-1}\\
 =&\int_{\partial B_1}\big(|\nabla w_r|^2 +2w_r\big) - 2(n+2)\int_{\partial B_1}w_r^2\,d\mathcal{H}^{n-1}\\
 =&(n+2)\int_{B_1} (|\nabla w_r|^2 + 2w_r)\,d\mathcal{H}^{n-1}- 2(n+2)\int_{\partial B_1}w_r^2\,d\mathcal{H}^{n-1}= (n+2)\Psi_{w_r}(1).
\end{split}
\end{equation*}
Therefore, we conclude that 
\begin{equation}\label{Phi'> in lem}
 \Phi'(r)\geq \frac{(n+2)}{r}\big(\Psi_{w_r}(1)-\Phi(r)\big) + \frac{1}{r}\int_{\partial B_1} \Big(\big(\langle\nu,\nabla u_r\rangle-2u_r\big)^2\,d\mathcal{H}^{n-1} - C\,\left(r^{-\frac{n}{\Theta}}+\frac{\omega(r)}{r}\right). 
\end{equation}
We can also note that, being $w_r$ the $2$-homogeneous extension of ${u_r}_{|\partial B1}$, thanks to 
\eqref{cambio variabili w_r} and \eqref{cond inf lemma tecn}, there exists a global solution $w$ of type $(A)$ such that
\begin{equation*}
 \|w_r-w\|_{H^1(B_1)}\leq \frac{1}{\sqrt{n+2}}\|{w_r}_{\partial B_1} - w\|_{H^1(\partial B_1)} \leq \d.
\end{equation*}
Hence, we can apply the epiperimetric inequality \eqref{dis epip} to $w_r$ and find a function $\xi\in w_r + H^1_0(B_1)$
such that 
\begin{equation}\label{epip in lem}
 \Psi_\xi(1)-\theta\leq (1-k)\big(\Psi_{w_r}(1)-\theta \big).
\end{equation}
Moreover, we can assume without loss of generality (otherwise we substitute $\xi$ with $u_r$) that $\Psi_\xi(1)\leq \Psi_{u_r}(1)$. 
Then, from the minimality of $u_r$ in $\mathcal{E}$ with respect to its boundary conditions \eqref{abuso_notazione}
and lemma \ref{lem reg mu} we have that
\begin{equation}\label{Psi_xi > in lem}
 \begin{split}
  \Psi_\xi(1)=& \int_{B_1}\big(|\nabla \xi|^2 + 2\xi\big)\,dx -\int_{\partial B_1} \xi^2\,d\mathcal{H}^{n-1}\\
  \geq&  \int_{B_1}\big(\langle \A(rx)\nabla \xi,\nabla \xi\rangle + 2\,f(rx)\xi\big)\,dx -\int_{\partial B_1} \mu(rx)\,\xi^2\,d\mathcal{H}^{n-1}\\
  &-C\,\left(r^{1-\frac{n}{p^*}}+\omega(r)\right) \int_{B_1}\big(|\nabla \xi|^2 + 2\xi\big)\,dx - C\,r^\c \int_{\partial B_1} \xi^2\,d\mathcal{H}^{n-1}\\
  \geq& \Phi(r) - C\,\left(r^{1-\frac{n}{p^*}}+\omega(r)\right) \int_{B_1}\big(|\nabla \xi|^2 + 2\xi\big)\,dx - C\,r^\c \int_{\partial B_1} \xi^2\,d\mathcal{H}^{n-1}\\
  \geq& \Phi(r) - C\,\left(r^{1-\frac{n}{p^*}}+\omega(r)\right).
 \end{split}
\end{equation}
From \eqref{epip in lem} and \eqref{Psi_xi > in lem} we get
\begin{equation}\label{Psi-Phi in lem}
 \begin{split}
 \Psi_{w_r}(1)-\Phi(r)&\geq \frac{1}{1-k}\left(\Phi(r)-\theta-C\left(r^{1-\frac{n}{p^*}}+\omega(r)\right)\right) + \theta - \Phi(r) \\
 &= \frac{k}{1-k} (\Phi(r)-\theta)- C\left(r^{1-\frac{n}{p^*}}+\omega(r)\right).
 \end{split}
\end{equation}
Then from \eqref{Phi'> in lem} and \eqref{Psi_xi > in lem}
\begin{equation}\label{lemmatecn Phi'>}
 \Phi'(r)\geq \frac{n+2}{r}\frac{k}{1-k} (\Phi(r)-\theta) - C\,\left(r^{-\frac{n}{\Theta}}+\frac{\omega(r)}{r}\right). 
\end{equation}
Let now $\widetilde{C}_6\in (0,\,(1-\frac{n}{\Theta}) \wedge (n+2)\frac{k}{1-k})$, then
\begin{equation}\label{Phi-theta)'}
 \bigg((\Phi(r)-\theta)\,r^{-\widetilde{C}_6}\bigg)'\geq - C\,\left(r^{-\frac{n}{\Theta}-\widetilde{C}_6}+\frac{\omega(r)}{r^{1+\widetilde{C}_6}}\right).
\end{equation}
Indeed, by taking into account \eqref{lemmatecn Phi'>}
\begin{equation*}
\begin{split}
 \Big((&\Phi(r)-\theta)\,r^{-\widetilde{C}_6}\Big)' =  \Phi'(r)r^{-\widetilde{C}_6} - \widetilde{C}_6\,(\Phi(r)-\theta)\,r^{-\widetilde{C}_6-1}\\
 &\geq \bigg(\frac{n+2}{r}\frac{k}{1-k} (\Phi(r)-\theta) - C\,\left(r^{-\frac{n}{\Theta}}+\frac{\omega(r)}{r}\right)\bigg)r^{-\widetilde{C}_6} - \widetilde{C}_6\,(\Phi(r)-\theta)\,r^{-\widetilde{C}_6-1}\\
 &\geq(\Phi(r)-\theta)r^{-\widetilde{C}_6-1} \bigg((n+2)\frac{k}{1-k}-\widetilde{C}_6\bigg) - C\,\left(r^{-\frac{n}{\Theta}}+\frac{\omega(r)}{r}\right)r^{-\widetilde{C}_6}\\
 &\geq - C\,\left(r^{-\frac{n}{\Theta}-\widetilde{C}_6}+\frac{\omega(r)}{r^{1+\widetilde{C}_6}}\right).
\end{split}
\end{equation*}

By integrating \eqref{Phi-theta)'} in $(t,r_0)$ with $t\in (s_0,r_0)$ and multiplying by $t^{\widetilde{C}_6}$ we finally get 
 \begin{align*}
  &t^{\widetilde{C}_6}\,\bigg[(\Phi(r)-\theta)\,r^{-\widetilde{C}_6}\bigg]_t^{r_0}\geq - C\,t^{\widetilde{C}_6}\,\int_t^{r_0}\left(r^{-\frac{n}{\Theta}-\widetilde{C}_6}+\frac{\omega(r)}{r^{1+\widetilde{C}_6}}\right)\,dr
 \end{align*}
whence  
 \begin{equation}\label{Phi - theta}
 \begin{split}
  \Phi(t)-\theta &\leq C\bigg(\int_t^{r_0}\left(r^{-\frac{n}{\Theta}-\widetilde{C}_6}+\frac{\omega(r)}{r^{1+\widetilde{C}_6}}\right)\,dr + 1\bigg)t^{\widetilde{C}_6}\\
  &\leq C \bigg(r^{1-\frac{n}{\Theta}}+ t^{\widetilde{C}_6} + t^{\widetilde{C}_6} \int_t^{r_0}\frac{\omega(r)}{r^{1+\widetilde{C}_6}}\,dr  \bigg) \leq C\, t^{\widetilde{C}_6} \left(\int_t^{r_0}\frac{\omega(r)}{r^{1+\widetilde{C}_6}}\,dr + 1\right).
 \end{split}
 \end{equation}
Consider now $\varrho_0<\varrho<r_0$  and estimate as follows
\begin{equation*}
 \begin{split}
  \int_{\partial B_1}|u_t-&u_\varrho|\,d\mathcal{H}^{n-1}=\int_{\partial B_1}\left|\int_\varrho^t\frac{d}{dr}\left(\frac{u(rx)}{r^2}\right)\,dr\right|\,d\mathcal{H}^{n-1}\\
  \leq& \int_\varrho^t\!\! r^{-2}\int_{\partial B_1}\!\!\! \left|\langle\nabla u(rx),x \rangle - 2\frac{u(rx)}{r}\right|\,d\mathcal{H}^{n-1}\,dr
  = \int_\varrho^t\!\! r^{-1}\int_{\partial B_1} \!\!\left|\langle\nabla u_r(x),x \rangle - 2\,u_r(x)\right|\,d\mathcal{H}^{n-1}\,dr\\
  \leq& \sqrt{n\omega_n} \int_\varrho^t r^{-\frac{1}{2}}\left(r^{-1}\int_{\partial B_1}\left|\langle\nabla u_r(x),x \rangle - 2\,u_r(x)\right|^2\,d\mathcal{H}^{n-1}\right)^\frac{1}{2}\,dr.
 \end{split}
\end{equation*}
Combining \eqref{weiss stima Phir}, \eqref{Phi'> in lem}, \eqref{Psi-Phi in lem}, \eqref{Phi - theta} and the H\"{o}lder inequality 
we have
\begin{equation}\label{u_t-u_s in lem}
\begin{split}
 \int_{\partial B_1} |u_t &-u_\varrho|\,d\mathcal{H}^{n-1}\leq C\int_\varrho^t r^{-\frac{1}{2}}\bigg(\Phi'(r)+ C\left(r^{-\frac{n}{\Theta}}+\frac{\omega(r)}{r}\right)\bigg)^\frac{1}{2}\,dr\\
 &\leq C \left(\log\frac{t}{\varrho}\right)^\frac{1}{2}\,\left(\Phi(t)-\Phi(\varrho) + C\left(t^{1-\frac{n}{\Theta}}-\varrho^{1-\frac{n}{\Theta}}+ \int_\varrho^t\frac{\omega(r)}{r}\,dr\right)\right)^\frac{1}{2}\\
 &\leq C \left(\log\frac{t}{\varrho}\right)^\frac{1}{2}\,\left((\Phi(t)-\theta) +(\theta-\Phi(\varrho)) + C\left(t^{1-\frac{n}{\Theta}}+\int_\varrho^t\frac{\omega(r)}{r}\,dr\right)\right)^\frac{1}{2}\\
 &\leq C \left(\log\frac{t}{\varrho}\right)^\frac{1}{2}\,\left(t^{\widetilde{C}_6} + t^{\widetilde{C}_6} \int_t^{r_0}\frac{\omega(r)}{r^{1+\widetilde{C}_6}}\,dr + \int_0^t\frac{\omega(r)}{r}\,dr + \int_\varrho^t\frac{\omega(r)}{r}\,dr \right)^\frac{1}{2}.
\end{split}
\end{equation}
The function $|\log t|^{a}$ is decreasing if $t\in (0,1]$ and it is easy to prove that $t^{\widetilde{C}_6}|\log t|^{a}\searrow 0$. 
If $r_0<<1$ then we have
\begin{equation*}
t^{\widetilde{C}_6} \int_t^{r_0}\frac{\omega(r)}{r^{1+\widetilde{C}_6}}\,dr 
+ \int_0^t\frac{\omega(r)}{r}\,dr\leq |\log t|^{-a}\int_0^{r_0}\frac{\omega(r)|\log t|^{a}}{r}\,dr.
\end{equation*}
Therefore thanks to the hypothesis $(H4)$ with $a>2$, 
if $r_0<<1$ for every $0\leq t\leq r_0$, then we achieve $(\omega(t)\vee t^{C_6})\leq |\log t|^{-a}$. 
Then the following holds   
\begin{equation}\label{u_t-u_s < log in lem}
\begin{split}              
\int_{\partial B_1} |u_t -u_\varrho|\,d\mathcal{H}^{n-1} &\leq C \left(\log\frac{t}{\varrho}\right)^\frac{1}{2}\,
	|\log t|^{-\frac{a}{2}}\left(1 +\int_0^{r_0}\frac{\omega(r)\,|\log r|^a}{r}\,dr\right)^\frac{1}{2}
					   \leq C \left(\log\frac{t}{\varrho}\right)^\frac{1}{2}\,|\log t|^{-\frac{a}{2}}. 
\end{split}
\end{equation}
A simple dyadic decomposition argument then leads to the conclusion. If $\varrho\in [2^{-k},2^{-k+1})$ 
and $t\in [2^{-h},2^{-h+1})$ with $h<k$, applying \eqref{u_t-u_s < log in lem}  
\begin{equation*}
 \int_{\partial B_1} |u_t-u_\varrho|\,d\mathcal{H}^{n-1}\leq C\,\sum_{j=h}^k \log(2^{j})^{-\frac{a}{2}}\leq C_7\, \sum_{j=h}^\infty \frac{1}{j^{\frac{a}{2}}}=: C_7\, \rho(t),
\end{equation*}
with 
\begin{equation}\label{rho}
 \rho(t):=\sum_{j=h}^\infty \frac{1}{j^{\frac{a}{2}}}\quad \textrm{if} \quad t\in [2^{-h},2^{-h+1}).
\end{equation}
By taking \eqref{H3'} into account  we have $a>2$, therefore, the function $\rho(t)$ is growing and infinitesimal in $0$, 
from which the conclusion of the lemma follows.
 \end{proof}

Checking the hypothesis of Lemma \ref{lem tecn unic blow-up} 
it is possible to prove the uniqueness of the blow-ups at regular points of the free-boundary:   
\begin{prop}[{\cite[Proposition 4.10]{FGS}}]\label{prop unicita blow-up pti regolari}
Let $u$ be a solution to the obstacle problem \eqref{PDE_u} with $f$ that satisfies $(H4)$ with $a>2$ and $x_0\in Reg(u)$. Then, there exist constants $r_0=r_0(x_0), \eta_0=\eta_0(x_0)$
such that every $x\in \Gamma_u\cap B_{\eta_0}(x_0)$ is a regular point and, denoting by $v_{x}=w(\LL^{-1}(x)y)$ any blow-up of  
$u$ in $x$ we have
\begin{equation}\label{stima unicità blow-up pti reg}
 \int_{\partial B_1}|u_{\LL(x),r}-w|\,d\mathcal{H}^{n-1}(y)\leq C_7\,\rho(r) \qquad \forall\,r\in (0,r_0), 
\end{equation}
where $C_7$ is an independent constant from  $r$ and $\rho(r)$ a growing, infinitesimal function in $0$. 
In particular, the blow-up limit $v_x$ is unique.
\end{prop}

\begin{oss}
If $f$ is $\a$-H\"{o}lder we can prove Lemma \ref{lem tecn unic blow-up} and Proposition \ref{prop unicita blow-up pti regolari} 
with $\rho(t)=t^{C_6}$ where $C_6:=\frac{\bar{C}_6\wedge \a}{2}$.
\end{oss}

\section{Regularity of the free-boundary}
In this last section we state some regularity results of the free-boundary of $u$, the solution of \eqref{problema ad ostacolo}.
If the matrix $\A$ satisfies the hypotheses $(H1)$-$(H2)$ and the linear term $f$ satisfies the hypothesis $(H4)$ with $a>2$
we obtain differentiability of the free-boundary in a neighborhood of any point $x\in Reg(u)$. In particular if $f$ is H\"{o}lder
we establish the $C^{1,\b}$ regularity as in \cite{FGS} where $\A$ is Lipschitz continuous.

Since the arguments involved are those used in \cite{FGS} together with the preliminary assumption developed in the previous sections
we do not provide any proof.
\begin{teo}[{\cite[Theorem 4.12]{FGS}}]\label{teo regolarità pti reg}
 Assume hypotheses $(H1)$, $(H2)$ and $(H4)$ with $a>2$ hold. Let $x\in Reg(u)$. Then, there exists $r>0$ such that $\Gamma_u\cap B_r(x)$ 
 is hypersurface $C^1$ and $n$ its normal vector is absolutely continuous with modulus of continuity depending on $\rho$ 
 defined in \eqref{rho}.
 
 In particular if $f$ is H\"{o}lder continuous there exists $r>0$ such that $\Gamma_u\cap B_r(x)$ 
 is hypersurface  $C^{1,\b}$ for some universal exponent $\b\in (0,1)$.
\end{teo}

We are able to say less on the set of singular points. We know that under the hypotheses $(H1)$, $(H2)$ and $(H4)$ wiht $a\geq 1$, 
the set $Sing(u)$ is contained in the union of $C^1$ submanifold. 

\begin{defi}\label{d:singular stratum}
 The singular stratum $S_k$ of dimension $k$ for $k=0,1,\dots,n-1$ is the subset of points $x\in Sing(u)$ for which 
 $\mathrm{Ker}(\mathbb{B}_x)=k$.
\end{defi}

In the following theorem we show that the set $Sing(u)$ has a stronger regularity property than rectifiabilty: 
we show that the singular stratum $S_k$ is locally contained in a single submanifold.
Moreover that $\cup_{k=l}^{n-1} S_k$ is a closed set for every $l=0,1,\dots,n-1$.

\begin{teo}[{\cite[Theorem 4.14]{FGS}}]\label{teo regolarita pti sing} 
 Assume hypotheses $(H1)$, $(H2)$ and $(H4)$ with $a\geq 1$. Let $x\in S_k$. Then there exists $r$ such that $S_k\cap B_r(x)$ 
 is contained in regular $k$-dimensional submanifold of $\R^n$.
\end{teo}

\subsection*{Acknowledgments} The author is grateful and would like to thank Professor Matteo Focardi for his suggestion of the problem 
 and for his constant support and encouragement. The author is partially supported by project GNAMPA $2015$ 
 ``Regolarità per problemi di analisi geometrica e del calcolo delle variazioni''.
 The author is member of the Gruppo Nazionale per l'Analisi Matematica, la Probabilità e le loro Applicazioni (GNAMPA) of
the Istituto Nazionale di Alta Matematica (INdAM).

{\footnotesize

\vspace{0.5cm}
\noindent\textsc{DiMaI, Università degli studi di Firenze}\\
\emph{Current address:} Viale Morgagni 67/A, 50134 Firenze (Italy)\\
\emph{E-mail address:} \texttt{geraci@math.unifi.it}
}

\end{document}